\documentclass{amsart}
\usepackage{verbatim, amsmath, amsfonts}
\usepackage{tikz, enumitem, theoremref}
\usetikzlibrary{arrows.meta}

\newtheorem{theorem}{Theorem}
\newtheorem{proposition}[theorem]{Proposition}
\newtheorem{lemma}[theorem]{Lemma}
\newtheorem{corollary}[theorem]{Corollary}

\newtheorem*{theorem*}{Theorem}

\theoremstyle{definition}

\newtheorem{definition}[theorem]{Definition}

\newcommand{\spe}[1]{#1}

\DeclareMathOperator{\Aut}{Aut}
\DeclareMathOperator{\Des}{Des}
\DeclareMathOperator{\Cyc}{Cyc}
\DeclareMathOperator{\Fix}{Fix}

\DeclareMathOperator{\rev}{rev}

\DeclareMathOperator{\vspan}{span}

\DeclareMathOperator{\des}{des}
\DeclareMathOperator{\asc}{asc}

\DeclareMathOperator{\ps}{ps}
\DeclareMathOperator{\cyc}{cyc}

\DeclareMathOperator{\sgn}{sgn}

\DeclareMathOperator{\qsym}{QSYM}
\newcommand{\qrep}{C(\mathfrak{G}, \qsym)}
\newcommand{\chifullevaluate}[4]{\Omega(#1,#2,#3;#4) }
\newcommand{\chifull}[3]{\Omega(#1,#2,#3) }
\newcommand{\chiexevaluate}[1]{\chifullevaluate{\spe{#1}}{\mathfrak{G}}{\mathbf{x}}{\mathfrak{g}}}
\newcommand{\chiex}[1]{\chifull{\spe{#1}}{\mathfrak{G}}{\mathbf{x}}}
\newcommand{\chievaluate}{\chiexevaluate{D}}
\newcommand{\chih}{\chiex{D}}

\newcommand{\chialphafullevaluate}[4]{\chi_{#1}(#2,#3;#4) }
\newcommand{\chialphafull}[3]{\chi_{#1}(#2,#3) }
\newcommand{\chialphaexevaluate}[1]{\chialphafullevaluate{\alpha}{\spe{#1}}{\mathfrak{G}}{\mathfrak{g}}}
\newcommand{\chialphaex}[1]{\chialphafull{\alpha}{\spe{#1}}{\mathfrak{G}} }

\newcommand{\genpoly}[4]{#1_{#2}(#3, #4)}

\newcommand{\genq}[4]{#1_{#2}(#3, #4, \mathbf{x} )}

\newcommand{\ggq}[2]{\ggq{#1}{\mathfrak{G}}{#2}{n}}

\begin{document}

\title{The Chromatic Quasisymmetric Class Function of a Digraph}

\author{Jacob A.~White}
\address{School of Mathematical and Statistical Sciences\\
University of Texas - Rio Grande Valley\\
Edinburg, TX 78539}
\keywords{Chromatic Polynomials, Quasisymmetric Functions, Poset partitions, Group actions} 

\subjclass{05E05, 05E18}
\date{\today}


\begin{abstract}

We introduce a quasisymmetric class function associated to a group acting on a double poset or on a directed graph. The latter is a generalization of the chromatic quasisymmetric function of a digraph introduced by Ellzey, while the latter is a generalization of a quasisymmetric function introduced by Grinberg. We prove representation-theoretic analogues of classical and recent results, including $F$-positivity, and combinatorial reciprocity theorems. We also deduce results for orbital quasisymmetric functions. We also study a generalization of the notion of strongly flawless sequences.

\end{abstract}
\maketitle

\section{Introduction}

Given a graph $G$, let $\mathfrak{G}$ be a subgroup of the automorphism group of $G$. Then $\mathfrak{G}$ acts on the set of $k$-colorings of $G$. If we $\chi_{\mathfrak{G}}(G,k)$ denote the number of orbits of this action, then the resulting function is a polynomial in $k$, called the \emph{orbital chromatic polynomial} and studied by Cameron and Kayibi \cite{cameron-kayibi}. Jochemko \cite{jochemko} found a combinatorial reciprocity theorem by giving a combinatorial interpretation to $(-1)^n \chi_{\mathfrak{G}}(G,-k)$.

Similarly, given a poset $P$, a subgroup $\mathfrak{G}$ of the automorphism group of $P$ acts on the set of order-preserving maps $\varphi: P \to \{1, \ldots, k \}$. If we let $\Omega_{\mathfrak{G}}(P,k)$ denote the number of orbits of this action, we obtain the orbital order polynomial that was introduced by Jochemko \cite{jochemko}, who proved a combinatorial reciprocity theorem for $\Omega_{\mathfrak{G}}(P,k)$. These results were later generalized to quasisymmetric functions associated to a double poset by Grinberg \cite{grinberg}. One of our primary interests is proving combinatorial reciprocity theorems for orbital polynomial invariants associated to combinatorial objects. 

 Stapledon \cite{stapledon} studied the equivariant Ehrhart quasipolynomial of a polytope. Let $\mathfrak{G}$ be a finite group acting linearly on a lattice $M'$ of rank $n$, and let $P$
be a $d$-dimensional $\mathfrak{G}$-invariant lattice polytope. Let $M$ be a translation of the
intersection of the affine span of $P$ and $M'$
to the origin, and consider the induced
representation $\rho: \mathfrak{G} \to GL(M)$ If $\chi(m)$ is the permutation character associated to the action of $G$ on the lattice of points in the $m$th dilate of $P$, then $\chi(m)$ is a quasipolynomial in $m$ whose coefficients are elements of $R(\mathfrak{G})$, the ring of virtual characters of $\mathfrak{G}$. Stapledon proved several results concerning the equivariant Ehrhart quasipolynomial.

Motivated by these past results, we study quasisymmetric class functions. These are class functions associated to the symmetry group $\mathfrak{G}$ of a combinatorial object, whose values are quasisymmetric functions. Equivalently, they are quasisymmetric functions whose coefficients are class functions.
If we let $\mathfrak{G}$ be the trivial group, then we obtain ordinary quasisymmetric functions (and should re-derive classical results). In fact, we can always obtain the ordinary quasisymmetric function by evaluating all characters at the identity element. We can also obtain corresponding orbital quasisymmetric functions, and various polynomial specializations.

The goal of this paper is to study a quasisymmetric class function generalization of the $\spe{D}$-partition enumerator of a double poset and of the chromatic polynomial of a directed graph. The former is an class function generalization of an invariant introduced by Grinberg, which in turn is a generalization of the labeled $P$-partition enumerator studied by Gessel \cite{gessel}. We define double posets and related terminology in Section \ref{sec:dbl}.
Given a double poset $D$ on a set $N$, let $\mathfrak{G}$ be a subset of the automorphism group of $D$. Then $\mathfrak{G}$ acts on the set of $D$-partitions. For $\mathfrak{g} \in \mathfrak{G}$, we define \[\Omega(D, \mathfrak{G}, \mathbf{x}; \mathfrak{g}) = \sum\limits_{\sigma: \mathfrak{g} \sigma = \sigma} \prod_{v \in N} x_{\sigma(v)}\] where we are summing over $D$-partitions fixed by $\mathfrak{g}$.
Then $\Omega(D, \mathfrak{G}, \mathbf{x})$ is a $\qsym$-valued class function.

Stanley introduced the chromatic symmetric function \cite{stanley-coloring-1}, a symmetric function generalization of the chromatic polynomial. This has been generalized to a chromatic quasisymmetric function by Shareshian and Wachs \cite{shareshian-wachs} and to directed graphs by Ellzey \cite{ellzey}. We will study a class function generalization of Ellzey's invariant, defined more explicitly in Section \ref{sec:digraph}. Much like the generalization of Shareshian and Wachs, our invariant has an extra variable $t$: our invariant is a class function that takes on values in the ring of quasisymmetric functions over the field $\mathbb{C}(t)$.
Given a digraph $G$ on a set $N$, let $\mathfrak{G}$ be a subset of the automorphism group of $G$. Then $\mathfrak{G}$ acts on the set of proper colorings of the underlying undirected graph. Ellzey defines a statistic $\asc(f)$ for a coloring. We show that this statistic is also $\mathfrak{G}$-invariant. For $\mathfrak{g} \in \mathfrak{G}$, we define
\[\chi(G, \mathfrak{G}, \mathbf{x}; \mathfrak{g}) = \sum\limits_{f: \mathfrak{g} f = f} t^{\asc(f)} \prod_{v \in N} x_{\sigma(v)}\]
where we sum over proper colorings of $G$. Much like in the case of double posets, the resulting invariant is a class function whose values are quasisymmetric functions over $\mathbb{C}[t]$.

Our primary interest is to study generalizations of $F$-positivity results, inequalities, and combinatorial reciprocity theorems. Let $C(\mathfrak{G}, \qsym)$ be the set of class functions with values in $\qsym$.
If we take an element $\chi(\mathbf{x})$ of $\qrep$ and a given basis $B$ for quasisymmetric functions, we say that $\chi(\mathbf{x})$ is \emph{$B$-effective} if $\chi(\mathbf{x})$ can expressed in the basis $B$ with coefficients that are characters of representations of $\mathfrak{G}$. If we let $\chi_1, \ldots, \chi_k$ denote the irreducible characters of $\mathfrak{G}$, then $R = \{\chi_i B_{\alpha}: i \in [k], B_{\alpha} \in B \}$ forms a basis for $C(\mathfrak{G}, \qsym)$. If $\chi(\mathbf{x})$ is $B$-effective, then $\chi(\mathbf{x})$ can be expressed as a linear combination of $\chi_i B_{\alpha}$ with nonnegative integer coefficients. 

We prove that the $D$-partition quasisymmetric class function for locally special posets is $F$-effective in Theorem \ref{thm:dblposfeffect}.
The notion of \emph{locally special} was first introduced by Grinberg, under the name tertispecial. He also suggests locally special as an alternative name. Our results specialize to both known and new results in the literature. We also give a proof that the corresponding orbital $D$-partition enumerator is $F$-positive. This implies that locally special double posets have $F$-positive $D$-partition enumerators, which appears to be new. It also implies $F$-positivity for skew Schur functions and for labeled $P$-partition enumerators.

We also prove in Theorem \ref{thm:feffectivedigraph} that $\chi(G, \mathfrak{G}, \mathbf{x})$ is also $F$-effective.

We study polynomial invariants as well. There are lots of results about inequalities for coefficients of chromatic polynomials of graphs with respect to different bases, including recent work that the coefficients of $(-1)^n\chi(G, -x)$ are unimodal \cite{huh} and strongly flawless \cite{kubitzke}.
Given a sequence $(f_0, \ldots, f_d)$, we say the sequence 
 is \emph{strongly flawless} if the following inequalities are satisfied:
\begin{enumerate}
    \item for $0 \leq i \leq \frac{d-1}{2}$, we have $f_i \leq f_{i+1}$.
    \item For $0 \leq i \leq \frac{d}{2}$, we have $f_i \leq f_{d-i}$.
\end{enumerate}
For this paper, we are focused on the sequence of coefficients for a polynomial $p(x)$ with respect to the basis $\binom{x}{k}$. We refer to these coefficients as the $f$-vector, and say $p(x)$ is \emph{strongly flawless} if the $f$-vector is strongly flawless and nonnegative. We have a representation-theoretic generalization: now the $f_i$ are required to be effective characters, and we interpret inequalities of the form $f_i \leq f_k$ as saying that $f_k - f_i$ is also a character. We refer to such a sequence of characters as \emph{effectively flawless}.
We show that $\Omega(D, \mathfrak{G}, x)$ and $\chi(G, \mathfrak{H}, x)$ are effectively flawless in Section \ref{sec:increasing}.
Then we obtain the following theorem:

We also discuss combinatorial reciprocity theorems. In \cite{stanley-crt}, he defines a combinatorial reciprocity theorem as `a result which establishes a kind of duality between two enumeration problems'. The book by Beck and Sanyal \cite{beck-sanyal} is full of many examples of such results. In general, we suppose that we have a vector subspace $V$ of a ring of formal power series, and that $V$ comes equipped with an involution $\omega$. Given two generating functions $f, g \in V$, a \emph{combinatorial reciprocity theorem} is the statement that $f = \omega g$. This is more general than the examples that appear in Beck and Sanyal's book, but still fits the general notion Stanley originally proposed.

In this paper, $V$ is usually vector space of class functions from $\mathfrak{G}$ to quasisymmetric functions of a fixed degree $d$, and $\omega = (-1)^d S \sgn$. Hence
a combinatorial reciprocity theorem for a quasisymmetric class function consists of showing that $(-1)^d S \sgn p(\mathbf{x})$ is $M$-realizable by giving an explicit description of the resulting permutation characters. The $\sgn$ term is the sign representation, which naturally arises as $\mathfrak{G}$ is always a a permutation group. It arises naturally in the work of Stapledon, Grinberg and Jochemko. We are also able to deduce combinatorial reciprocity theorems for corresponding orbital invariants, and for polynomial invariants as well.

We prove a combinatorial reciprocity theorem for double posets in Theorem \ref{thm:combrecdbl}, which involves taking duals of partial orders, and a combinatorial reciprocity theorem for digraphs in Theorem \ref{thm:combrecdig}, which involves group actions on pairs $(O,f)$, where $O$ is an acyclic orientation and $f$ is a compatible coloring.
 
The paper is organized as follows. In Section \ref{sec:prelim}, we define quasisymmetric functions, review some representation theory, and discuss set compositions. We also discuss polynomials, and quasisymmetric class functions. In Section \ref{sec:dbl}, we define double posets, $D$-partitions and the corresponding $D$-partition quasisymmetric class function.  Then we prove some basic facts about $\Omega(D, \mathfrak{G}, \mathbf{x})$. We also discuss some properties about locally special double posets that we need for later proofs. In Section \ref{sec:digraph}, we define the chromatic quasisymmetric class function, and provide a formula expressing $\chi(G, \mathfrak{G}, \mathbf{x})$ in terms of quasisymmetric class functions related to double posets coming from acyclic orientations of $G$. In Section \ref{sec:crt}, we prove our combinatorial reciprocity theorems for $\Omega(D, \mathfrak{G}, \mathbf{x})$ and $\chi(G, \mathfrak{G}, \mathbf{x})$. In Section \ref{sec:increasing}, we show our polynomial invariants are effectively flawless. We also show other properties about the quasisymmetric functions, and study some examples to show how properties fail for $h$-vectors. In Section \ref{sec:feffect}, we prove $F$-effectiveness for $\Omega(D, \mathfrak{G}, \mathbf{x})$ and $\chi(G, \mathfrak{G}, \mathbf{x})$. We also establish the corresponding $h$-effectiveness for the related polynomial invariants, and deduce some $F$-positivity results as corollaries. In Section \ref{sec:orbinv}, we define our orbital quasisymmetric functions, and deduce facts about these invariants from the results we have obtained about the quasisymmetric class functions. Finally, in Section \ref{sec:future}, we discuss some open problems.

\section{Preliminaries}
\label{sec:prelim}

Given a basis $B$ for a vector space $V$, and $\vec{\beta} \in B, \vec{v} \in V$, we let $[\vec{\beta}] \vec{v}$ denote the coefficient of $\vec{\beta}$ when we expand $\vec{v}$ in the basis $B$.

Let $\mathbf{x} = x_1, x_2, \ldots $ be a sequence of commuting indeterminates. Let $n \in \mathbb{N}$ and let $f \in \mathbb{K}[[\mathbf{x}]]$ be a homogeneous formal power series in $\mathbf{x}$, where the degree of every monomial in $f$ is $n$. Then $f$ is a \emph{quasisymmetric function} if it satisfies the following property:
for every $S = \{i_1, \ldots, i_k\}$ with $i_1 < i_2 < \cdots < i_k$, and every integer composition $\alpha_1 + \cdots + \alpha_k = n$, we have $[\prod_{j=1}^k x_{i_j}^{\alpha_j}]f = [\prod_{j=1}^k x_j^{\alpha_j}]f$. Often, we will define quasisymmetric functions that are generating functions over functions. Given a function $f: S \to \mathbb{N}$, we define $\mathbf{x}^f = \prod_{v \in S} x_{f(v)}$. For example, the chromatic symmetric function of a graph $G$ is defined as $\sum\limits_{f:V \to \mathbb{N}} \mathbf{x}^f$ where the sum is over all proper colorings of $G$.

Given an integer composition $\alpha = (\alpha_1, \alpha_2, \ldots, \alpha_k)$ of $n$, we let \[M_{\alpha} = \sum\limits_{i_1 < \cdots < i_k} \prod_{j=1}^k x_{i_j}^{\alpha_j}.\] These are the \emph{monomial quasisymmetric functions}, which form a basis for the ring of quasisymmetric functions.

The second basis we focus on is Gessel's basis of fundamental quasisymmetric functions.
The set of integer compositions is partially ordered by refinement. With respect to this partial order, the set of integer compositions forms a lattice.
The \emph{fundamental quasisymmetric function} $F_{\alpha}$ are defined by:
\[ F_{\alpha} = \sum\limits_{\beta \geq \alpha} M_{\beta}. \]

There is a well-known bijection between subsets of $[n-1]$ of size $k-1$ and integer compositions $\alpha \models n$ of length $k$, given by defining $S(\alpha) = \{\alpha_1, \alpha_1+\alpha_2, \ldots, \alpha_1+\alpha_2+\ldots + \alpha_{k-1} \}$. Under this bijection, the lattice of integer compositions is isomorphic to the Boolean lattice. Given a set $A = \{s_1, \ldots, s_{k-1} \}$ with $s_1 < s_2 < \cdots < s_{k-1}$, we have $S^{-1}(A) = (s_1, s_2-s_1, s_3-s_2, \ldots, s_k-s_{k-1}, n-s_k)$.

There is an important linear transformation on quasisymmetric functions called the \emph{antipode}: 
\[ S(M_{\alpha}) = (-1)^{\ell(\alpha)} \sum\limits_{\beta \leq \alpha} M_{\overleftarrow{\beta}} \]
where $\overleftarrow{\beta}$ is the composition given by reversing the order of $\beta$. Antipodes exist for any graded connected bialgebra, and are analogous to inversion for groups.

Our proofs rely a lot on working with set compositions, and quasisymmetric functions related to set compositions.
Given a finite set $N$, a \emph{set composition} is a sequence $(S_1, \ldots, S_k)$ of disjoint non-empty subsets whose union is $N$. We denote set compositions as $S_1|S_2|\cdots|S_k$, and refer to the sets $S_i$ as \emph{blocks}. We use $C \models N$ to denote that $C$ is a set composition of $N$, and let $\ell(C) = k$ be the length of the composition. Given $C$, the associated integer composition is $\alpha(C) = (|C_1|, |C_2|, \ldots, |C_k|)$. We refer to $\alpha(C)$ as the \emph{type} of $C$. We partially order set compositions by refinement. 
Finally given a set composition $C$ of type $\beta$ and $\alpha \leq \beta$, let $C_{\alpha}(C)$ be the unique set composition of type $\alpha$ such that $C_{\alpha}(C) \leq C$.

\subsection{Group actions and class functions}

 Given a group action $\mathfrak{G}$ on a set $X$, we let $X / \mathfrak{G}$ denote the set of orbits. For $x \in X$, $\mathfrak{G}_x$ is the stabilizer subgroup, and $\mathfrak{G}(x)$ is the orbit of $x$. Also, a \emph{transversal} is a subset $T \subset X$ such that $|T \cap O| = 1$ for every orbit $O$ of $X$. Finally, for $\mathfrak{g} \in \mathfrak{G}$, we let $\Fix_{\mathfrak{g}}(X) = \{x \in X: \mathfrak{g}x = x \}$.

There is an action of $\mathfrak{S}_N$ on the collection of all set compositions of $N$. Given a permutation $\mathfrak{g} \in \mathfrak{S}_N$, and a set composition $C \models N$, we let  \[\mathfrak{g} C = \mathfrak{g}(C_1)| \mathfrak{g}(C_2) | \cdots | \mathfrak{g}(C_k).\] Then we obtain an action of $\mathfrak{S}_N$ on the collection of all set compositions of $N$.

We assume familiarity with representation theory of finite groups - see \cite{fulton-harris} for basic definitions. Recall that, given any group action of $\mathfrak{G}$ on a finite set $X$, there is a group action on $\mathbb{C}^X$ as well, which gives rise to a representation. The resulting representations are called \emph{permutation representations}. We are working with representations over $\mathbb{C}$. We let $C(\mathfrak{G})$ be the ring of class functions of $\mathfrak{G}$. There is an orthonormal basis of $C(\mathfrak{G})$ given by the characters of the irreducible representations of $\mathfrak{G}$. We refer to elements $\chi \in C(\mathfrak{G})$ that are integer combinations of characters as \emph{virtual characters}, and elements that are nonnegative integer linear combinations as \emph{effective characters}. Finally, we see $\chi$ is a \emph{permutation character} if it is the character of a permutation representation. We partially order $C(\mathfrak{G})$ by saying $\chi \leq_{\mathfrak{G}} \psi$ if $\psi - \chi$ is an effective character.

Let $R$ be a $\mathbb{C}$-algebra. Let $C(\mathfrak{G}, R)$ be the set of class functions from $\mathfrak{G}$ to $R$. That is, for every $\mathfrak{g}, \mathfrak{h} \in \mathfrak{G}$, and $\chi \in C(\mathfrak{G}, R)$, we have $\chi(\mathfrak{hg}\mathfrak{h}^{-1}) = \chi(\mathfrak{g}).$ For our paper, $R$ is usually $\qsym$ or $\mathbb{C}[x]$. 

Let $\mathbf{B}$ be a basis for $R$. For $b \in \mathbf{B}$, $\mathfrak{g} \in \mathfrak{G}$, and $\chi \in C(\mathfrak{G}, R)$, let $\chi_b(\mathfrak{g}) = [b]\chi(\mathfrak{g})$. Then $\chi_b$ is also a class function. Thus we can write $\chi = \sum\limits_{b \in \mathbf{B}} \chi_b b$. Conversely, given a family $\chi_b$ of class functions, one for each $b \in \mathbf{B}$ the function $\chi$ defined by $\chi(\mathfrak{g}) = \sum\limits_{b \in \mathbf{B}} \chi_b(\mathfrak{g}) b$ is a class function in $C(\mathfrak{G}, R)$.
We say that $\chi$ if $\mathbf{B}$-effective if $\chi_b$ is an effective character for all $b \in \mathbf{B}$. We say that $\chi$ is $\textbf{B}$-realizable if $\chi_b$ is a permutation character for all $b$.
If $\mathbf{B}$ has a partial order on it, then we say $\chi$ is $\mathbf{B}$-increasing if for all $b \leq c$ in $\mathbf{B}$, we have $\chi_b \leq_{\mathfrak{G}} \chi_c$. Assuming $\chi$ is $\mathbf{B}$-effective, this is equivalent to saying that $\chi_c$ is the character for the representation of $\mathfrak{G}$ on some module $V$, and $\chi_b$ is the character of a representation of a submodule of $V$.

A \emph{quasisymmetric class function} is an element of $\qrep$.
\begin{proposition}
Let $\chi \in \qrep$ have degree $d$. If $\chi$ if $F$-effective, then $\chi$ is $M$-increasing.
\end{proposition}
\begin{proof}
Let $\chi = \sum_{\alpha \models d} \psi_{\alpha} F_{\alpha}$. Then there exists $\mathfrak{G}$-modules $W_{\alpha}$ such that $\psi_{\alpha}$ is the character of the representation of $\mathfrak{G}$ on $W_{\alpha}$. If we let $V_{\alpha} = \bigoplus_{\beta \leq \alpha} W_{\beta}$, Then $V_{\alpha}$ has character $\sum_{\beta \leq \alpha} \psi_{\beta} = [M_{\alpha}] \chi$. 

Let $\alpha \leq \beta$. Then we see that $V_{\alpha}$ is a submodule of $V_{\beta}$. Hence $\chi_{\beta} - \chi_{\alpha}$ is the character of the complement of $V_{\alpha}$ in $V_{\beta}$. Thus $\chi$ is $M$-increasing.
\end{proof}

Given a subgroup $\mathfrak{H}$ of $\mathfrak{G}$, and a 
class function $\chi \in C(\mathfrak{H}, R)$, We define the \emph{induced class function} $\chi\uparrow_{\mathfrak{H}}^{\mathfrak{G}} \in C(\mathfrak{G}, R)$ by \[\chi\uparrow_{\mathfrak{H}}^{\mathfrak{G}}(\mathfrak{g}) = \frac{1}{|\mathfrak{H}|} \sum\limits_{\mathfrak{k} \in \mathfrak{G}: \mathfrak{kg}\mathfrak{k}^{-1} \in \mathfrak{H}} \chi(\mathfrak{kg}\mathfrak{k}^{-1}).\]

Finally, we define a function $\langle \cdot, \cdot \rangle: C(\mathfrak{G}, R) \times C(\mathfrak{G}, R) \to R$ by 
$\langle \chi, \psi \rangle = \frac{1}{|\mathfrak{G}|} \overline{\chi}(\mathfrak{g}) \psi(\mathfrak{g})$ where $\overline{x}$ is the complex conjugate. In the case where $R = \mathbb{C}$, this is the usual inner product on class functions.

\begin{proposition}
\label{prop:global} 
Let $\mathfrak{G}$ be a finite group, let $R$ be a $\mathbb{C}$-algebra with basis $\mathbf{B}$. Fix $\chi \in C(\mathfrak{G}, R)$.
\begin{enumerate}[label={(\arabic*)},itemindent=1em]
\item For $b \in \mathbf{B}$, we have $[b]\left(\chi\uparrow_{\mathfrak{H}}^{\mathfrak{G}}\right) = \left([b]\chi\right)\uparrow_{\mathfrak{H}}^{\mathfrak{G}}$.
\item Given an irreducible character $\psi$, we have $\langle \chi, \psi \rangle = \sum\limits_{b, c \in B} \langle \chi_b, \psi_c \rangle b \cdot c$.
\item If $\chi$ is $\mathbf{B}$-effective, and $\psi$ is an irreducible character, then $\langle \psi, \chi \rangle$ is $\mathbf{B}$-positive. \label{prop:charpositive}
\item Suppose $\textbf{B}$ is partially ordered. Let $\psi \in C(\mathfrak{G})$. If $\chi$ is $B$-increasing, then for all $b \leq c$ in $\textbf{B}$ we have $[b]\langle \psi, \chi \rangle \leq [c]\langle \psi, \chi \rangle.$  \label{prop:charincreasing}
\end{enumerate}
\end{proposition}
\begin{proof}
Let $\mathfrak{g} \in \mathfrak{G}$. Then
\begin{align*} \chi\uparrow_{\mathfrak{H}}^{\mathfrak{G}}(\mathfrak{g}) & = \frac{1}{|\mathfrak{H}|} \sum\limits_{\mathfrak{k} \in \mathfrak{G}: \mathfrak{kg}\mathfrak{k}^{-1} \in \mathfrak{H}} \chi(\mathfrak{kg}\mathfrak{k}^{-1}) \\
& = \frac{1}{|\mathfrak{H}|} \sum\limits_{\mathfrak{k} \in \mathfrak{G}: \mathfrak{kg}\mathfrak{k}^{-1} \in \mathfrak{H}} \sum\limits_{b \in B} \chi_b(\mathfrak{kg}\mathfrak{k}^{-1}) b \\ 
& = \sum\limits_{b \in B} \left( \frac{1}{|\mathfrak{H}|} \sum\limits_{\mathfrak{k} \in \mathfrak{G}: \mathfrak{kg}\mathfrak{k}^{-1} \in \mathfrak{H}} \chi_b(\mathfrak{kg}\mathfrak{k}^{-1})\right) b \\
& = \sum\limits_{b \in B} \chi_b \uparrow_{\mathfrak{H}}^{\mathfrak{G}}(\mathfrak{g}) b.
\end{align*} 
Thus we see that the first result follows from comparing the coefficient of $b$ on both sides.

Let $\psi \in C(\mathfrak{G}, R)$. Then
\begin{align*}
\langle \chi, \psi \rangle & = \frac{1}{|\mathfrak{G}|} \sum\limits_{\mathfrak{g} \in \mathfrak{G}} \bar{\chi}(\mathfrak{g}) \psi(\mathfrak{g}) \\
& = \frac{1}{|\mathfrak{G}|} \sum\limits_{\mathfrak{g} \in \mathfrak{G}} \left(\sum\limits_{b \in B} \bar{\chi_b}(\mathfrak{g}) b \right) \left(\sum\limits_{c \in B} \psi_c(\mathfrak{g}) c \right) \\
& = \sum\limits_{b \in B} \sum\limits_{c \in B} \frac{1}{|\mathfrak{G}|} \sum\limits_{\mathfrak{g} \in \mathfrak{G}}  \bar{\chi_b}(\mathfrak{g}) \psi_c(\mathfrak{g}) b \cdot c.
\end{align*}

For the third claim, let $\psi$ be an irreducible character. Let $b \in \mathbf{B}$. Using the second claim, we have $[b] \langle \psi, \chi \rangle = \langle \psi, [b]\chi \rangle \geq 0.$ Hence $\langle \psi, \chi \rangle$ is $B$-positive.

Now suppose that $\chi$ is $\textbf{B}$-increasing. Let $b \leq c \in \mathbf{B}$. Then there is a representation of $\mathfrak{G}$ whose character is $\rho := [c] \chi - [b] \chi$. Using the second claim, we see that \[[c]\langle \psi, \chi \rangle - [b] \langle \psi, \chi \rangle = \langle \psi, [c] \chi - [b] \chi \rangle = \langle \psi, \rho \rangle \geq 0.\] Hence $[c]\langle \psi, \chi \rangle \geq [b]\langle \psi, \chi \rangle$.
\end{proof}

\subsection{Principal specialization}

Given a polynomial $p(x)$ of degree $d$, define $h(t) = (1-t)^{d+1}\sum_{m \geq 0} p(m) t^m$. The sequence of coefficients of $h(t)$ is the \emph{$h$-vector} of $p(x)$. 
We define the $f$-vector $(f_0, \cdot, f_d)$ via $p(x) = \sum_{i=0}^d f_i \binom{x}{i}$. We say that $p(x)$ is \emph{strongly flawless} if the following inequalities are satisfied:
\begin{enumerate}
    \item for $0 \leq i \leq \frac{d-1}{2}$, we have $f_i \leq f_{i+1}$.
    \item For $0 \leq i \leq \frac{d}{2}$, we have $f_i \leq f_{d-i}$.
\end{enumerate}
There is a lot of interest in log-concave and unimodal sequences in combinatorics. We consider strongly flawless sequences to also be interesting, as strongly flawless unimodal sequences can be seen as a generalization of symmetric unimodal sequences. Examples of results with strongly flawless sequences include the work of Hibi \cite{hibi} and Juhnke-Kubitzke and Van Le \cite{kubitzke}.

Given a quasisymmetric function $F(\mathbf{x})$ of degree $d$, there is an associated polynomial $\ps(F)(x)$ given by principal specialization. For $x \in \mathbb{N}$, we set \[x_i = \begin{cases} 1 & i \leq x \\ 0 & i > x
\end{cases} \]
The resulting sequence is a polynomial function in $x$ of degree $d$, which we denote by $ps(F)(x)$.
If we write $F(\mathbf{x}) = \sum\limits_{\alpha \models d} c_{\alpha} M_{\alpha}$, then $f_i = \sum\limits_{\alpha \models d: \ell(\alpha) = i} c_{\alpha}$.
Similarly, if we write $F(\mathbf{x}) = \sum\limits_{\alpha \models n} d_{\alpha} F_{\alpha}$, then $h_i = \sum\limits_{\alpha \models d: \ell(\alpha) = i} d_{\alpha}$.

The set $\mathbb{K}[x]$ is a Hopf algebra, with antipode given by $Sp(x) = p(-x)$. Also, $\varphi: \qsym \to \mathbb{K}[x]$ given by $\varphi(F(\mathbf{x})) = \ps(F)(x)$ is a Hopf algebra homomorphism and $\varphi(SF(\mathbf{x})) = f(-x)$. 

Let $F(\mathbf{x}) \in \qsym$ be of degree $n$, and $\mathfrak{g} \in \mathfrak{G}$. Define $\ps(F) \in C(\mathfrak{G}, \mathbb{C}[x])$ by $\ps(F)(\mathfrak{g}) = \ps(F(\mathbf{x}; \mathfrak{g})).$ Then $\ps(F)$ is also the principal specialization, resulting in an polynomial class function. If we write $\ps(F) = \sum\limits_{i=0}^d f_i \binom{x}{i}$, then $(f_0, \ldots, f_d)$ is the \emph{equivariant $f$-vector} of $\ps(F)$, which consists of permutation characters.
If we write $\sum\limits_{m \geq 0} \ps(F) t^m = \frac{h(t)}{(1-t)^n}$, then the coefficients of $h(t)$ is the \emph{equivariant $h$-vector} of $\ps(F)$. Note that the entries of the equivariant $h$-vector are virtual characters. We say $\ps(F)$ is $h$-effective if the entries are effective characters. We say that $\ps(F)$ is \emph{effectively flawless} if we have the following system of inequalities:
\begin{enumerate}
    \item for $0 \leq i \leq \frac{d-1}{2}$, we have $f_i \leq_{\mathfrak{G}} f_{i+1}$.
    \item For $0 \leq i \leq \frac{d}{2}$, we have $f_i \leq_{\mathfrak{G}} f_{d-i}$.
\end{enumerate}

We can obtain results about polynomial class functions from the corresponding quasisymmetric class functions.
\begin{proposition}
\label{prop:global2}
Let $F(\mathbf{x}) \in \qsym$ be of degree $d$, and $\mathfrak{g} \in \mathfrak{G}$.  
\begin{enumerate}[label={(\arabic*)},itemindent=1em]
    \item If we write $F(\mathbf{x}) = \sum\limits_{\alpha \models d} \chi_{\alpha} M_{\alpha}$, then $\ps(F) = \sum\limits_{i=0}^d \chi_{\alpha} \binom{x}{\ell(\alpha)}$.
    \item If we write $F(\mathbf{x}) = \sum\limits_{\alpha \models d} \psi_{\alpha} F_{\alpha}$, then $h(t) = \sum\limits_{\alpha \models d} \psi_{\alpha}t^{\ell(\alpha)}$. If $F(\mathbf{x})$ if $F$-effective, then $\ps(F)$ is $h$-effective. \label{prop:ftoheffect}
    \item If $G(\mathbf{x}) \in \qrep$ with $(-1)^d \sgn S F(\mathbf{x}) = G(\mathbf{x}),$ then \[(-1)^d \sgn \ps(F)(-x) = \ps(G)(x).\]
    \item If $F(\mathbf{x})$ is $M$-realizable and $M$-increasing, then $\ps(F)$ is effectively flawless. \label{prop:ftohorbit}
    \item Let $\psi$ be an irreducible character. If $F(\mathbf{x})$ is $F$-effective, then $\langle \psi, \ps(F) \rangle$ is $h$-positive. If $F(\mathbf{x})$ is $M$-increasing, then $\langle \psi, \ps(F) \rangle$ is strongly flawless. \label{prop:flawlessorb}
\end{enumerate}
\end{proposition}
\begin{proof}
The first three results are proven in a similar manner.
Let $\mathfrak{g} \in \mathfrak{G}$. Then $\ps(F)(x; \mathfrak{g}) = \ps(F(\mathbf{x}; \mathfrak{g}))(x)$. Since $F(\mathbf{x}, \mathfrak{g}) = \sum\limits_{\alpha \models d} \chi_{\alpha}(\mathfrak{g}) M_{\alpha}$, we have \[\ps(F(\mathbf{x}; \mathfrak{g})) = \sum\limits_{\alpha \models d} \chi_{\alpha}(\mathfrak{g}) \binom{x}{\ell(\alpha)}.\] The result follows.

For the fourth result, let $d$ be the degree of $F(\mathbf{x})$. For each $\alpha \models d$, let $V_{\alpha}$ be a $\mathfrak{G}$-module with character $[M_{\alpha}]F(\mathbf{x})$. Since $F(\mathbf{x})$ is $M$-increasing, we know there exists injective $\mathfrak{G}$-invariant functions $\theta_{\alpha, \beta}: V_{\alpha} \to V_{\beta}$ for every $\alpha \leq \beta \models d$. 
We let $V_i = \bigoplus\limits_{\alpha \models d: \ell(\alpha) = i} V_{\alpha}$. Then the character of $V_i$ is $f_i$. To show the inequalities, it suffices to find $\mathfrak{G}$-invariant injections between $V_i$ and $V_j$. Then $f_j - f_i$ is the character of the complement of $V_i$ in $V_j$.

We need to recall that the boolean lattice, and hence the lattice of integer compositions, has a symmetric chain decomposition, a result due to DeBruijn \cite{debruijn}. Let $C(d)$ be the set of integer compositions of $d$. A symmetric chain decomposition is a partition of $C(d)$ into saturated chains $c_1, \ldots, c_m$ with the property that, for each chain $c_i$, the sum of the ranks of the first and last element of $c_i$ is $d$. 

Fix a symmetric chain decomposition $D$. Fix integers $i$ and $j$ such that $1 \leq i < j \leq d - i$. Consider an integer composition $\alpha$ with $\ell(\alpha) = i$. Then there exists a chain $x_1 < x_2 < \cdots < x_k$ in $D$ with $x_{i-\ell(x_1)+1} = \alpha$. Define $\varphi_{i,j}(\alpha) = x_{j-\ell(x_1)+1}$.
We see that the following two facts are true:
\begin{enumerate}
    \item If $ i \leq \frac{d-1}{2}$, then $\varphi_{i,i+1}$ is injective.
    \item If $i < \frac{d}{2}$, then $\varphi_{i,d-i}$ is a bijection.
\end{enumerate}

Let $1 \leq i \leq \frac{d-1}{2}$, and let $\alpha \models d$ with $\ell(\alpha) = i$. We define $\theta_{i,j}: V_i \to V_j$ by requiring $\theta_{i,j}|_{V_{\alpha}} = \theta_{\alpha, \varphi_{i,j}(\alpha)}$. Then $\theta_{i,j}(V_{\alpha}) \subseteq V_{\varphi_{i,j}(\alpha)}$. 

Thus $\theta_{i,j}$ is an injective $\mathfrak{G}$-invariant map.
Thus $V_i$ is isomorphic to a submodule of $V_{i+1}$, and we have $f_i \leq_{\mathfrak{G}} f_{i+1}$.

Now let $i \leq \frac{d}{2}$.
Let $\theta_i: V_i \to V_{d-i}$ be given by $\theta_i|_{V_{\alpha}} = \theta_{\alpha, \varphi_{i, d-i}(\alpha)}$. By a similar argument, $\theta_i$ is injective and $\mathfrak{G}$-invariant. Hence $V_i$ is isomorphic to a submodule of $V_{d-i}$, and $f_i \leq_{\mathfrak{G}} f_{d-i}$.

For the last result, let $\psi$ be an irreducible character. A simple calculation shows that $\langle \psi, \ps(F) \rangle = \ps( \langle \psi, F \rangle )$. If $F(\mathbf{x})$ is $F$-effective, then $\langle \psi, F \rangle$ is $F$-positive. Since the entries of the $h$-vector are non-negative sums of coefficients in the $F$-basis, the $h$-vector of $\ps( \langle \psi, F, \rangle)$ is non-negative.

Finally, suppose that $F(\mathbf{x})$ is $M$-increasing. Viewing $\langle \psi, F(\mathbf{x}) \rangle$ as a quasisymmetric class function for the trivial group. Then $\langle \psi, F \rangle$ is $M$-increasing. Hence $\ps(\langle \psi, F \rangle)$ is effectively flawless. Since we are working with the trivial group, we conclude that $\langle \psi, \ps(F) \rangle$ is strongly flawless.

\end{proof}

\section{Double Posets}
\label{sec:dbl}

Now we will discuss double posets. The Hopf algebra of double posets was introduced by Malvenuto and Reutenauer \cite{malvenuto}. Grinberg associated a quasisymmetric function to any double poset, which is a generalization of Gessel's $P$-partition enumerator. This quasisymmetric function is studied extensively by Grinberg \cite{grinberg}, who proved a combinatorial reciprocity theorem.

Given a finite set $N$, a \emph{double poset} on $N$ is a triple $(N, \leq_1, \leq_2)$ where $\leq_1$ and $\leq_2$ are both partial orders on $N$. Often for standard poset terminology, we will use $\leq_i$ as a prefix to specify which of the two partial orders is being referred to. For instance, a $\leq_1$-order ideal is a subset that is an order ideal with respect to the first partial order, and a $\leq_1$-covering relation refers to a pair $(x,y)$ such that $x \prec_1 y$.

\begin{figure}
\begin{center}
\begin{tabular}{cc}
\begin{tikzpicture}
  \node[circle, draw=black, fill=white] (b) at (0,2) {$b$};
  \node[circle, draw=black, fill=white] (a) at (0,0) {$a$};
  \node[circle, draw=black, fill=white] (c) at (2,0) {$c$};
  \node[circle, draw=black, fill=white] (d) at (2,2) {$d$};
 \draw[-Latex] (b) -- (a);
  \draw[-Latex] (d) edge (c);
  \draw[-Latex] (b) -- (c);
    \draw[-Latex] (d) -- (a);

\end{tikzpicture}

& 

\begin{tikzpicture}
  \node[circle, draw=black, fill=white] (b) at (0,2) {$c$};
  \node[circle, draw=black, fill=white] (a) at (0,0) {$b$};
  \node[circle, draw=black, fill=white] (c) at (2,0) {$d$};
  \node[circle, draw=black, fill=white] (d) at (2,2) {$a$};
 \draw[-Latex] (b) -- (a);
  \draw[-Latex] (d) edge (c);

\end{tikzpicture}
\end{tabular}
\end{center}
\label{fig:doubleposet}
\caption{A double poset.}
\end{figure}
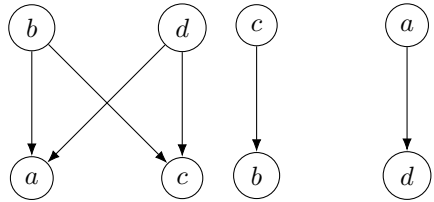

 Let $D$ be a double poset on a finite set $N$, and let $f: N \to \mathbb{N}$. Then $f$ is a $D$-partition if and only if it satisfies the following two properties:
\begin{enumerate}
    \item For $i \leq_1 j$ in $D$, we have $f(i) \leq f(j)$.
    \item For $i \leq_1 j$ and $j \leq_2 i$ in $D$, we have $f(i) < f(j)$.
\end{enumerate}

Let $P_D$ be the set of $D$-partitions.
We define the $D$-partition enumerator by
\begin{equation}
    \Omega_{D, \varphi}(D, \mathbf{x}) = \sum\limits_{f\in P_D} \prod_{v \in N} x_{f(v)}
    \label{eq:dpartition}
\end{equation}
This quasisymmetric function is studied extensively by Grinberg \cite{grinberg}.

Given a double poset $D$, a pair $(m,m') \in M$ is an \emph{inversion} if $m <_1 m'$ and $m' <_2 m$. Given a set composition $C \models N$, we say that $C$ is a $D$-set composition if it satisfies the following two properties:
\begin{enumerate}
    \item For every $i$, $C_1 \cup C_2 \cup \cdots \cup C_i$ is a $\leq_1$-order ideal.
    \item For every $i$, there are no inversions in $C_i$
\end{enumerate}
Let $X_{D}$ be the set of $D$-set compositions.
\begin{proposition}
Let $D$ be a double poset. Then $\Omega(D, \mathbf{x}) = \sum\limits_{C \in X_{D}} M_{\alpha(C)}.$
\label{prop:oldcoloring}
\end{proposition}

Given a double poset $D$, an automorphism is a bijection $\sigma: N \to N$ such that, for all $x, y \in N$ and all $i \in \{1,2\}$, we have $x \leq_i y$ if and only if $\sigma(x) \leq_i \sigma(y)$. We let $\Aut(D)$ be the automorphism group of $D$. For instance, for the double poset in Figure \ref{fig:doubleposet}, the permutation $(ac)(bd)$ is the only nontrivial automorphism. Similarly the only nontrivial automorphism of the double poset in Figure \ref{fig:specialposet} is the permutation $(a)(bd)(c)$.

Let $\mathfrak{G} \subseteq \Aut(D)$. For $\mathfrak{g} \in \mathfrak{G}$, define \[\Omega(D, \mathfrak{G}, \mathbf{x}; \mathfrak{g}) =  \sum\limits_{f \in \Fix_{\mathfrak{g}}(P_D)} \mathbf{x}^f. \]
This is the \emph{$D$-partition quasisymmetric class function}.


Naturally, there is an \emph{order polynomial class function} as well: given a positive integer $n$, we let $X_{n, D}$ be the set of $D$-partitions $\sigma: D \to [n]$. Then $\mathfrak{G}$ acts on $X_{n, D}$ and we let $\Omega(D, \mathfrak{G}, n)$ be the resulting character.

We give two alternative formulas for $\chievaluate$, and another formula for the order polynomial class function. Let $X_{\alpha, D}$ be the set of $D$-set compositions of type $\alpha$. Then $\mathfrak{G}$ acts on $X_{\alpha, D}$. Let $\chi_{\alpha, D}$ be the resulting character.
\begin{theorem}
Let $D$ be a double poset on a finite set $N$ and let  $\mathfrak{G} \subseteq \Aut(D)$.
Then we have the following identities:
\begin{enumerate} 
\item \[\chievaluate =  \sum\limits_{C \in \Fix_{\mathfrak{g}}(X_{D})} M_{\alpha(C)}  \]
\item \[\Omega(D, \mathfrak{G}, \mathbf{x}) = \sum\limits_{\alpha \models |N|} \chi_{\alpha, D} M_{\alpha} \]
\item \[\Omega(D, \mathfrak{G}, x) = \sum\limits_{\alpha \models |N|} \chi_{\alpha, D}(D, \mathfrak{G}) \binom{x}{|\alpha|}.\]
\end{enumerate}

\label{thm:coloring}
\end{theorem}

\begin{figure}
\begin{center}
\begin{tabular}{cc}
\begin{tikzpicture}
\node[circle,draw=black] (A) at (0,1) {a};
\node[circle,draw=black ] (B) at (1,0) {b};
\node[circle,draw=black] (C) at (0,-1) {c};
\node[circle,draw=black] (D) at (-1,0) {d};

\draw[-Latex] (A) -- (B);
\draw[-Latex] (B) -- (C);
\draw[-Latex] (D) -- (C);
\draw[-Latex] (A) -- (D);
\end{tikzpicture}
&
\begin{tikzpicture}
  \node[circle, draw=black, fill=white] (b) at (0,2) {$b$};
  \node[circle, draw=black, fill=white] (a) at (0,0) {$a$};
  \node[circle, draw=black, fill=white] (c) at (2,0) {$c$};
  \node[circle, draw=black, fill=white] (d) at (2,2) {$d$};
 \draw[-Latex] (b) -- (a);
  \draw[-Latex] (d) edge (c);
  \draw[-Latex] (b) -- (c);
    \draw[-Latex] (d) -- (a);

\end{tikzpicture}
\end{tabular}
\end{center}
\label{fig:specialposet}
\caption{A double poset.}
\end{figure}
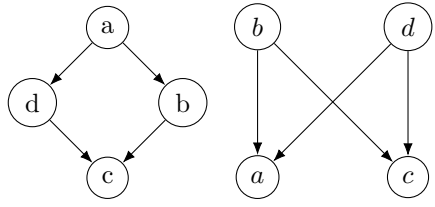

\begin{proof}
Fix a double poset $D$ on a finite set $N.$

For the first formula, let $f \in \Fix_{\mathfrak{g}}(P_D)$. Let $i_1 < i_2 < \cdots < i_k$ be the natural numbers for which $f^{-1}(i_j) \neq \emptyset$. Define $C(f) = f^{-1}(i_1)|f^{-1}(i_2)|\cdots|f^{-1}(i_k)$. This is the composition associated with $f$. We see that 
\[M_{\alpha(C)} = \sum\limits_{f \in \Fix_{\mathfrak{g}}(P_D)} \mathbf{x}^f \]
and that $\mathfrak{g}f = f$ if and only if $\mathfrak{g}C(f) = C(f)$. 
Thus we obtain
\[ \sum\limits_{C \in \Fix_{\mathfrak{g}}(X_D)} M_{\alpha(C)} = \sum\limits_{C \in \Fix_{\mathfrak{g}}(X_{D})} \sum\limits_{f \in \Fix_{\mathfrak{g}}(P_D)} \mathbf{x}^f =  \sum\limits_{f \in \Fix_{\mathfrak{g}}(P_D)} \mathbf{x}^f.\]

Given $\mathfrak{g}$, we have 
\[ \sum\limits_{C \in \Fix_{\mathfrak{g}}(X_{D})} M_{\alpha(C)} = \sum\limits_{\alpha \models |N|} \sum\limits_{C \in \Fix_{\mathfrak{g}}(X_{\alpha, D})} M_{\alpha} = \sum\limits_{\alpha \models |N|} \chi_{\alpha, D}(\mathfrak{g}) M_{\alpha}\]
and so by the first formula, we conclude that $\Omega(D, \mathfrak{G}, \mathbf{x}) = \sum\limits_{\alpha \models |N|} \chi_{\alpha, D} M_{\alpha}$.

For the third formula, let $\mathfrak{G}$ act trivially on $\binom{[n]}{k}$, the collection of $k$-subsets of $[n]$. Then for any integer composition $\alpha \models |N|$, we have an action of $\mathfrak{G}$ on $\binom{[n]}{\ell(\alpha)} \times X_{\alpha, D}$. Given a $D$-partition $f: D \to [n]$, let $\varphi(f) = (f(N), C(f))$. Then this defines an isomorphism of $\mathfrak{G}$-sets:
\[ X_{n, D} \simeq \bigcup_{\alpha \models |N|} \binom{[n]}{\ell(\alpha)} \times X_{\alpha, D} \]
The result follows from taking linear spans to obtain $\mathfrak{G}$-modules, and then taking the trace to obtain characters.
\end{proof}

As an example, consider the double poset $D$ in Figure \ref{fig:doubleposet}, and let $\mathfrak{G} = \Aut(D)$. Let $\rho$ denote the regular representation. Then \begin{align*} \Omega(D, \mathfrak{G}, \mathbf{x}) & = M_{2,2}+\rho(M_{1,1,2}+M_{1,1,1,1}) \\ 
& = F_{2,2}+\sgn (F_{1,1,2}+F_{2,1,1}-F_{1,1,1,1})+\rho F_{1,2,1}.
\end{align*}

As another example, consider the double poset $D$ in Figure \ref{fig:specialposet}, and let $\mathfrak{G} = \Aut(D)$. Let $\rho$ denote the regular representation. Then \begin{align*} \Omega(D, \mathfrak{G}, \mathbf{x}) & = M_{1,3}+M_{1,2,1}+\rho(M_{1,1,2}+M_{1,1,1,1}) \\ 
& = F_{1,3}+\sgn F_{1,1,2}.
\end{align*}

\subsection{Properties of Double Posets}

A double poset is \emph{locally special} if whenever $y$ $\leq_1$-covers $x$, then $x$ and $y$ are $\leq_2$-comparable. The double poset in Figure \ref{fig:specialposet} is locally special, while the double poset in Figure \ref{fig:doubleposet} is not.
Grinberg gives several examples of locally special posets, including double posets coming from skew shapes and labeled posets.

We say that an inversion pair $(x,y)$ is a \emph{descent pair} if $x \prec_1 y$.
\begin{lemma}
Let $D$ be a locally special double poset, and let $I \subseteq J$ be $\leq_1$-order ideals. If $D$ has an inversion pair $(x,y)$ with $x,y \in J \setminus I$, then $D$ has a descent pair $(w,z)$ with $w, z \in J \setminus I$.
\label{lem:inversion-to-descent}
\end{lemma}
\begin{proof}
We prove the result by induction on $|J \setminus I|$.
Let $x \prec_1 t \leq_1 z$. Since $D$ is locally special, we have $x \leq_2 t$ or $t \leq_2 x$. In the latter case, we have found a descent pair $(x,t)$. In the former case the pair $(t,y)$ forms an inversion pair. We observe that then interval $[t,y]$ is equal to $(y) \setminus (t)$, where $(a)$ is the principal order ideal generated by $a$. Since $|[t,y]| < |J \setminus I|$, by induction there is a descent pair $(w,z)$ in $[t,y]$, and we have $x \leq_1 \prec_1 z \leq y$.
\end{proof}

For any set $S \subseteq N$, and a partial order $P$ on $N$, we let $P|_S$ denote the induced poset on $S$. The same notation is also used for linear orders (which are a special case of partial orders), and for double posets.

We say that a linear order $\ell$ of $N$ is $D$-compatible if only if for all pairs $(I,J)$ of $\leq_1$-order ideals with $I \subseteq J$, the linear order $\ell|_{J \setminus I}$ is a $\leq_1$-linear extension of $D|_{J \setminus I}$ if and only if $D|_{J \setminus I}$ contains no inversion pairs.
\begin{lemma}
Let $D$ be a double poset on a finite set $N$. If $D$ is locally special, then there exists a $D$-compatible linear order.
\label{lem:compatible}
\end{lemma}
\begin{proof}
Let $D$ be a locally special double poset. Let $G(D)$ be the directed graph obtained by taking the directed edges of the Hasse diagram of $\leq_1$, and reversing the direction on edges $x \prec_1 y$ if $x >_2 y$. We claim that $G(D)$ is acyclic.
Suppose that we have a directed cycle $C$ in $G(D)$. Let $C$ have vertices $x_0, x_1, \ldots, x_k$ in order. Note that this means that $x_i \prec_1 x_{i+1}$ or $x_{i+1} \prec_1 x_i$ for all $i$. Since $D$ is locally special, we have $x_i \leq_2 x_{i+1}$ for all $i$, which is a contradiction. Thus there is no directed cycle.

We say $x \leq_{P} y$ if there is a directed path from $y$ to $x$ in $G(D).$ Let $\ell$ be a linear extension of $P$. We claim that $\ell$ is $D$-compatible.

Let $I \subseteq J$ be $\leq_1$-order ideals. Suppose that there are no inversions in $J \setminus I$. Then we see that $P|_{J \setminus I} = D|_{J \setminus I}$, and thus $\ell|_{J \setminus I}$ is a linear extension of $D|_{J \setminus I}$. Suppose instead there is an inversion pair $(x,y)$ in $J \setminus I$. By Lemma \ref{lem:inversion-to-descent}, we can choose $(x,y)$ to be a descent pair. Since $x \geq_2 y$, we have $(x,y)$ is a directed edge in $G(D)$, and thus $x \geq_{P} y$. Since $\ell$ is a linear extension of $P$, we have $x >_{\ell} y$. Therefore $\ell|_{J \setminus I}$ is not a linear extension of $D|_{J \setminus I}$.

\end{proof}
Given a $D$-compatible linear order $\ell$, we can lexicographically order any other linear order of $N$: given two linear orders $\pi$ and $\sigma$, consider the first $i$ where $\pi_i \neq \sigma_i$. Then we say $\pi <_{\ell} \sigma$ if $\pi_i <_{\ell} \sigma$. 
Let $C$ be a $D$-set composition for $D$. We let $\ell(C)$ be the lexicographically first total refinement of $C$. 
Finally, we can also totally preorder $X_{\alpha, D}$, the set of $D$-set compositions of type $\alpha$. Given $C, C' \in X_{\alpha, D}$, we say $C \leq_{\ell} C'$ if and only if $\ell(C) \leq_{\ell} \ell(C')$. Note that it is a preorder because it is possible for $\ell(C) = \ell(C')$.

We see that a $D$-set composition with only singleton blocks is a $\leq_1$-linear extension. We prove a proposition regarding when such linear extensions are increasing with respect to $\ell$.
We say that a $\leq_1$-linear extension $\pi$ is \emph{increasing} if $\pi_1 <_{\ell} \pi_2 <_{\ell} \cdots <_{\ell} \pi_n$. 
\begin{proposition}
Let $D$ be a locally special double poset, and let $\ell$ be a $D$-compatible linear order.
Let $I \subseteq J$ be $\leq_1$-order ideals of $D$. Then $D|_{J \setminus I}$ has an increasing $\leq_1$-linear extension $\sigma$ if and only if $D$ has no inversions in $J \setminus I$. In that case, $\ell|_{J \setminus I} = \sigma$, and $\sigma$ is lexicographically least.
\label{prop:increasing}
\end{proposition}
\begin{proof}
We prove the result by induction on $k = |J \setminus I|.$ Suppose that $D|_{J \setminus I}$ has an increasing $\leq_1$-linear extension $\sigma$. Then $D|_{J \setminus (I \cup \{\sigma_1 \})}$ also has a $\leq_1$-increasing linear extension. By induction, we see that $\sigma|_{J \setminus \{\sigma_1 \}} = \ell|_{J \setminus I \cup \{\sigma_1 \}}$. Similarly, if we let $J' = J \setminus \{\sigma_k \}$ and we see that $D|_{J' \setminus I}$ has an increasing $\leq_1$-linear extension, and thus $\sigma|_{\{\sigma_1, \ldots, \sigma_{k-1} \}} = \ell|_{J' \setminus I}$. Therefore, we have $\ell|_{J \setminus I} = \sigma$. Hence $\ell|_{J \setminus I}$ is a linear extension of $D|_{J \setminus I}$, and by definition of $D$-compatible order, this means that $J \setminus I$ does not contain any inversions.

Now we suppose that $J \setminus I$ has no inversions. Then $\ell$ restricted to $J \setminus I$ is a linear extension of $D|_{J \setminus I}$ with respect to $\leq_1$. Moreover, $\ell|_{J \setminus I}$ is increasing.  

Let $\sigma = \ell|_{J \setminus I}$.
Let $\tau$ be another increasing $\leq_1$-linear extension. Suppose $\tau_1 \neq \sigma_1$. Then $\tau_k = \sigma_1$ for some $k > 1$. However, then $\tau_1 >_{\ell} \tau_k$, and hence $\tau$ is not increasing. Thus $\tau_1 = \sigma_1$. By induction, we have $\tau|_{J \setminus \{\tau_1 \}}$ and $\sigma|_{J \setminus \{\sigma_1 \}}$ are both increasing $\leq_1$-linear extensions of $D|_{J \setminus (I \cup \{\tau_1\})}$, and hence are equal by induction. Thus $\sigma = \tau$.

\end{proof}

\section{Digraph coloring}
\label{sec:digraph}

We refer to directed graphs as digraphs. We require that there is at most one directed edge between any two vertices. An example appears in Figure \ref{fig:digraph}.

\begin{figure}
\begin{center}
\begin{tikzpicture}
\node[circle,draw=black] (A) at (0,2) {A};
\node[circle,draw=black ] (B) at (2,0) {B};
\node[circle,draw=black] (C) at (0,-2) {C};
\node[circle,draw=black] (D) at (-2,0) {D};

\draw[-Latex,thick] (A) -- (B);
\draw[-Latex, thick] (B) -- (C);
\draw[-Latex, thick] (C) -- (D);
\draw[-Latex, thick] (D) -- (A);
\end{tikzpicture}
\end{center}
\caption{A digraph}
\label{fig:digraph}
\end{figure}
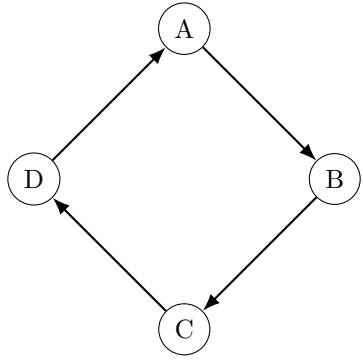
Given a digraph $G$ on $N$, a \emph{coloring} is a function $f: N \to \mathbb{N}$ which satisfies:
\begin{enumerate}
\item for every edge $(u,v)$, $f(u) \neq f(v)$.
\end{enumerate}
An edge $(u,v)$ is an $f$-ascent if $f(u) < f(v)$. We let $\asc(f)$ denote the number of $f$-descents.
An example of a coloring in Figure \ref{fig:digraph} is given by $f(A) = 1$, $f(B) = 2$, $f(C) = 3$, and $f(D) = 4$. This coloring has exactly one ascent, from $D$ to $A$. Let $C_G$ denote the set of all colorings, and we let $C_{n, G} = \{f \in C_G: f(N) \subseteq [n] \}.$
Finally, we let $C_{k, n, G} = \{f \in C_{n, G}: \asc(f) = k \}$.
\begin{definition}
The \emph{chromatic quasisymmetric function} is \[\genq{\chi}{}{G}{} = \sum\limits_{f \in C_G} t^{\asc(f)} \mathbf{x}^f.\] 

Likewise, for $n \in \mathbb{N}$, define the \emph{chromatic polynomial} to be \[\genpoly{\chi}{}{G}{n} = \sum\limits_{f \in C_{n,G}} t^{\asc(f)},\] where we sum over all proper colorings $f: I \to [n]$.
\end{definition}

For example, for the digraph $G$ in Figure \ref{fig:digraph}, we have $\chi(G, \mathbf{x}) = 2t^2M_{2,2}+4t^2M_{2,1,1}+4t^2M_{1,2,1}+4t^2M_{1,1,2}+(4t^3+16t^2+4t)M_{1,1,1,1}.$ For instance, $2M_{2,2}$ comes from colorings $f$ where $ f(A) = f(C)$ and $f(B) = f(D)$. In all such cases, there ends up being two ascents.

Now we define the automorphism group of a digraph. Given a digraph $G$ on a finite set $N$, a bijection $\mathfrak{g}: N \to N$ is an \emph{automorphism} if for every $u, v \in N$, we have $(u,v) \in E(G)$ if and only if $(\mathfrak{g}(u), \mathfrak{g}(v)) \in E(G)$.
Let $\Aut(G)$ be the set of automorphisms of $G$, which forms a group. For the digraph $G$ appearing in Figure \ref{fig:digraph}, the automorphism group is isomorphic to $C_4$, the cyclic group of order $4$, acting by rotations. 

Now we define the chromatic quasisymmetric class function. 
Let $\mathfrak{G} \subset \Aut(G)$. For $\mathfrak{g} \in \mathfrak{G}$, we let \[\chi(G, \mathfrak{G}, \mathbf{x}; \mathfrak{g}) = \sum\limits_{f \in \Fix_{\mathfrak{g}}(C_G)} t^{\asc(f)} \mathbf{x}^f\]
where the sum is over all proper colorings of $G$.
This defines a class function on $\mathfrak{G}$ whose values are quasisymmetric functions over $\mathbb{Q}[t]$. 
This is the \emph{chromatic quasisymmetric class function} associated to $G$.

As an example, consider the digraph in Figure \ref{fig:digraph}, and let $\mathfrak{G} = \mathbb{Z}/4\mathbb{Z}$ act via rotation. Let $\rho$ denote the regular representation. Then \[\chi(G, \mathfrak{G}, t, \mathbf{x}) = (1+\sgn)t^2 M_{2,2} + \rho t(tM_{2,1,1}+tM_{1,2,1}+tM_{1,1,2}) + (1+4t+t^2)M_{1,1,1,1}.\]

We let $\chi_i:\mathbb{Z}/4\mathbb{Z} \to \mathbb{C}$ be given by $\chi_i(j) = i^j$, and $\chi_{-i}:\mathbb{Z}/4\mathbb{Z} \to \mathbb{C} $ be given by $\chi_{-i}(j) = (-i)^j$. Then 
\begin{align*}\chi(G, \mathfrak{G}, t, \mathbf{x}) & = (1+\sgn)t^2 (F_{2,2}+3F_{1,1,1,1}) + t^2(\chi_i+\chi_{-i})(F_{2,1,1}+F_{1,1,2}) \\ & + \rho(tF_{1,1,1,1}+t^2F_{1,2,1}+t^3F_{1,1,1,1}). \end{align*}

We detail some formulas relating the chromatic quasisymmetric class function of a digraph $G$ to $D$-partition quasisymmetric class functions. The key concept for proving identities is an acyclic orientation.
For a directed graph $G$, an acyclic orientation is another digraph $O$ on the same vertex set, with no directed cycles, such that, for every $u, v \in N$, we have $(u,v) \in G$ if and only if $(u,v) \in E(O)$ or $(v,u) \in E(O)$. An $O$-ascent is an edge $(u,v) \in E(D)$ where $(u,v) \in E(O)$. We let $\asc(O)$ be the number of $O$-ascents. We let $\mathcal{A}(G)$ be the set of acyclic orientations. Given an acyclic orientation $O$, there is a natural double poset associated with $O$: for $x, y \in N$, we say $x \leq_1 y$ if and only if there is a directed path from $y$ to $x$ in $O$. Then we define $x \leq_2 y$ if and only if $y \leq_1 x$.

\begin{lemma}
Let $G$ be a directed graph and let $\mathfrak{G} \subseteq \Aut(G)$. 
\begin{enumerate}
\item For $\mathfrak{g} \in \mathfrak{G}$, we have \[ \chi(G, \mathfrak{G}, t, \mathbf{x}; \mathfrak{g}) = \sum\limits_{O \in \Fix_{\mathfrak{g}}(\mathcal{A}(G))} t^{\asc(O)} \Omega(P_O, \mathfrak{G}_{O}, \mathbf{x}; \mathfrak{g}).\]
\item We have \[ \chi(G, \mathfrak{G}, t, \mathbf{x}) = \sum\limits_{O \in \mathcal{A}(G)} \frac{t^{\asc(O)}}{|\mathfrak{G}(O)|}  \Omega(P_O, \mathfrak{G}_{O}, \mathbf{x})\uparrow_{\mathfrak{G}_{O}}^{\mathfrak{G}}. \]

\item Let $\mathcal{T}$ be a transversal for the group action of $\mathfrak{G}$ on $A(G)$. We have \[ \chi(G, \mathfrak{G}, \mathbf{x}) = \sum\limits_{O \in \mathcal{T}} t^{\asc(O)} \Omega(P_O, \mathfrak{G}_{O}, \mathbf{x})\uparrow_{\mathfrak{G}_{O}}^{\mathfrak{G}}. \]

\end{enumerate}
\label{lem:fundamental2}
\end{lemma}

\begin{proof}
Let $\mathfrak{g} \in \mathfrak{G}$, and let $f$ be a proper coloring of $G$. Let $O_f$ be the orientation of $G$ given by directing $v$ to $u$ if $f(v) > f(u)$. Then $O_f$ is an acyclic orientation. We also see that $\mathfrak{g}O_f = O_f$, and that $\asc(f) = \asc(O_f).$ Then $f$ is a $P_{O_f}$-partition. 
Thus \begin{align*} \chi(G, \mathfrak{G}, t, \mathbf{x}; \mathfrak{g}) & = \sum\limits_{f \in \Fix_{\mathfrak{g}}(C_G)} t^{\asc(f)} \mathbf{x}^f \\
& = \sum\limits_{O \in \Fix_{\mathfrak{g}}(A(G))} t^{\asc(O)}\sum\limits_{f \in \Fix_{\mathfrak{g}}(C_O): O_f = O} \mathbf{x}^f \\
& = \sum\limits_{O \in \Fix_{\mathfrak{g}}(A(G))} t^{\asc(O)} \Omega(P_O, \mathfrak{G}_O, \mathbf{x}; \mathfrak{g}). \end{align*}

To prove our second formula, we have
\begin{align*}
    & \sum\limits_{O \in A(G)} \frac{t^{\asc(O)}}{|\mathfrak{G}(O)|}\Omega(P_O, \mathfrak{G}_{O}, \mathbf{x})\uparrow_{\mathfrak{G}_O}^{\mathfrak{G}}(\mathfrak{g}) \\
    & = \sum\limits_{O \in A(G)} \frac{t^{\asc(O)}}{|\mathfrak{G}(O)|} \frac{1}{|\mathfrak{G}_{O}|} \sum\limits_{\mathfrak{h} \in \mathfrak{G}: \mathfrak{hg}\mathfrak{h}^{-1} \in \mathfrak{G}_{O}} \Omega(P_O, \mathfrak{G}_{O}, \mathbf{x}; \mathfrak{hg}\mathfrak{h}^{-1}) \\
     & = \frac{1}{|\mathfrak{G}|} \sum\limits_{\mathfrak{h} \in \mathfrak{G}} \sum\limits_{\substack{ O \in A(G) \\ \mathfrak{hg}\mathfrak{h}^{-1} \in \mathfrak{G}_{O}}} t^{\asc(O)}\Omega(P_O, \mathfrak{G}_{O}, \mathbf{x}; \mathfrak{hg}\mathfrak{h}^{-1}) \\
     & = \frac{1}{|\mathfrak{G}|} \sum\limits_{\mathfrak{h} \in \mathfrak{G}} \sum\limits_{\substack{O \in A(G) \\ \mathfrak{g} \in \mathfrak{G}_{\mathfrak{h}^{-1}O}}} t^{\asc(O)} \Omega(P_{\mathfrak{h}^{-1}O}, \mathfrak{G}_{\mathfrak{h}^{-1}O}, \mathbf{x}; \mathfrak{g}) \\
      & = \frac{1}{|\mathfrak{G}|} \sum\limits_{\mathfrak{h} \in \mathfrak{G}} \sum\limits_{\substack{O \in A(G)\\ \mathfrak{g} \in \mathfrak{G}_{O}}} t^{\asc(O)} \Omega(P_{O}, \mathfrak{G}_{O}, \mathbf{x}; \mathfrak{g}) \\
    &= \sum\limits_{O \in \Fix_{\mathfrak{g}}(\mathcal{A}(G))} t^{\asc(O)} \Omega(P_O, \mathfrak{G}_{O}, \mathbf{x}; \mathfrak{g}). 
\end{align*}
The first equality is a formula for computing induced characters. The second equality comes from the Orbit-Stabilizer Theorem, and changing the order of summation. The third equality is due to the fact that our quasisymmetric functions are invariants, so $\Omega(P_O, \mathfrak{G}_{O}, \mathbf{x}; \mathfrak{hg}\mathfrak{h}^{-1}) = \Omega(P_{\mathfrak{h}^{-1}O}, \mathfrak{G}_{\mathfrak{h}^{-1}O}, \mathbf{x}; \mathfrak{g}).$ Finally, we can reindex our summation by replacing $O$ with $\mathfrak{h}O$. Since $A(G)$ is invariant under $\mathfrak{G}$, we end up with the same number of terms in the sum. Since $\asc(O) = \asc(\mathfrak{g}O)$, we obtain the fourth equality. The last equality is then immediate.
\end{proof}

\section{Combinatorial Reciprocity Theorem}
\label{sec:crt}
As stated in the introduction, to us the setting of a combinatorial reciprocity theorem consists of a subspace $V$ of a ring of formal power series in some number of variables, and an involution $\omega$ on $V$. Then given two generating functions $f$ and $g$, a combinatorial reciprocity theorem is the claim that $f = \omega g$.

For example, let $V$ be the vector space of polynomial functions of degree $d$. Given a polynomial $p(x)$, we can consider the formal power series $\sum\limits_{n \geq 0} p(n) t^n$. In this way, $V$ is a subspace of the ring of formal power series in $t$. We also know that $V$ is part of a Hopf algebra, and so $\omega = (-1)^d S$ is an involution on $V$. This is the setting for many combinatorial reciprocity theorems in the literature.

Similarly, one can let $V$ be the vector space of quasisymmetric functions (or symmetric functions) of degree $d$, and let $\omega = (-1)^d S$, where $S$ is the antipode. There are several examples of combinatorial reciprocity theorems involving $\qsym$ or sym, including Theorem 4.2 of Stanley \cite{stanley-coloring-1} or Theorem 4.2 of Grinberg \cite{grinberg}.

\subsection{Combinatorial Reciprocity for Double Posets}

 Given a double poset $D$, the dual poset $D^{\ast}$ is obtained by reversing the first partial order $\leq_1$ of $D$.  We prove a combinatorial reciprocity theorem for locally special double posets. The orbital version of the combinatorial reciprocity theorem was previously obtained by Grinberg \cite{grinberg}. 
 We also prove a combinatorial reciprocity theorem for the chromatic quasisymmetric class function of a digraph.

\begin{theorem}
Let $D$ be a locally special double poset on a finite set $N$, and let $\mathfrak{G} \subseteq \Aut(D)$.

Then $(-1)^{|N|} \sgn S \Omega(D, \mathfrak{G}, \mathbf{x}) = \Omega(D^{\ast}, \mathfrak{G}, \mathbf{x})$.

Also, $(-1)^{|N|} \sgn \Omega(D, \mathfrak{G}, -x) = \Omega(D^{\ast}, \mathfrak{G}, x)$.
\label{thm:combrecdbl}
\end{theorem}
The second result follows from the first via principal specialization. Our proof relies a lot on the work of Grinberg \cite{grinberg}.

First, we discuss a weighted generalization of $\chiex{D}$. Given a weight function $\mathbf{w}: N \to \mathbb{N}$, a double poset $D$ on $N$, and a $D$-partition $\sigma$, define \[\mathbf{x}^{\mathbf{w}, \sigma} = \prod_{i \in N} x_{\sigma(i)}^{\mathbf{w}(i)}.\]

We let $\Omega(D, \mathbf{w}, \mathbf{x}) = \sum\limits_{\sigma} \mathbf{x}^{\mathbf{w}, \sigma}$ where the sum is over all $D$-partitions. The following is Theorem 4.2 of Grinberg \cite{grinberg}:
\begin{proposition}
Let $D$ be a double poset on a finite set $N$, and let $\mathbf{w}: N \to \mathbb{N}$. If $D$ is locally special, then
 \[(-1)^{|N|} S \Omega(D, \mathbf{w}, \mathbf{x}) = \Omega(D^{\ast}, \mathbf{w}, \mathbf{x}).\]
 \label{prop:grinberg}
\end{proposition}

We let $\Aut(D, \mathbf{w})$ be the set of automorphisms $\mathfrak{g}$ of $D$ with the property that $\mathbf{w} \circ \mathfrak{g} = \mathbf{w}$. Let $\mathfrak{G}$ be a subgroup of $\Aut(D, \mathbf{w})$. Then $\mathfrak{G}$ acts on the set of $D$-partitions.
We let $\chiexevaluate{D, w} = \sum\limits_{\sigma} \mathbf{x}^{\mathbf{w}, \sigma}$, where the sum is over all $D$-partitions $\sigma$ with the property that $\sigma \circ \mathfrak{g} = \sigma$. Since $\mathfrak{g} \in \Aut(G, \mathbf{w})$, the resulting power series is a quasisymmetric function.

Given $\mathfrak{g} \in \mathfrak{G}$, let $\Cyc(\mathfrak{g})$ be the set of cycles of $\mathfrak{g}$. We can define a new weight function $\mathbf{w}/\mathfrak{g}: \Cyc(\mathfrak{g}) \to \mathbb{N}$ by $\mathbf{w}/\mathfrak{g} (C) = \sum\limits_{x \in C} \mathbf{w}(x)$. Given two cycles $C_1$ and $C_2$ of $\Cyc(\mathfrak{g})$, and $i \in \{1,2\}$, we say $C_1 \leq_i C_2$ if there exists $x \in C_1$ and $y \in C_2$ such that $x \leq_i y$. This turns $\Cyc(\mathfrak{g})$ into a double poset that we denote by $D/\mathfrak{g}$. The following is a combination of parts of Propositions 7.5 and 7.6 of Grinberg \cite{grinberg}.
\begin{lemma}
Let $D$ be a double poset on a finite set $N$, let $\mathbf{w}: N \to \mathbb{N}$, and let $\mathfrak{G} \subseteq \Aut(D, \mathbf{w})$. Given $\mathfrak{g} \in \mathfrak{G}$, we have
\[\chiexevaluate{D, w} = \chiex{D/\mathfrak{g}, w/\mathfrak{g}}.\]
Moreover, if $D$ is locally special, then so is $D/\mathfrak{g}$.
\label{lem:grinberg}
\end{lemma}

Finally, we let $\mathbf{1}: N \to \mathbb{N}$ be the function defined by $\mathbf{1}(n) = 1$ for all $n \in N$.
\begin{proof}[Proof of Theorem \ref{thm:combrecdbl}]
We have
\begin{align*}
    (-1)^{|N|} S \sgn(\mathfrak{g})\chiexevaluate{D, 1} & = (-1)^{|N|} (-1)^{|N|-\cyc(\mathfrak{g})} S \chiex{D/\mathfrak{g}, 1/\mathfrak{g}} \\ 
    & = \chiex{(D/\mathfrak{g})^{\ast}, 1/\mathfrak{g}} \\
    & = \chiex{D^{\ast}/\mathfrak{g}, 1/\mathfrak{g}} \\
    & = \chiexevaluate{D^{\ast}, 1} 
\end{align*}
where the first and last equalities are due to Lemma \ref{lem:grinberg}, the second equality is due to Proposition \ref{prop:grinberg}.
\end{proof}

\subsection{Combinatorial Reciprocity Theorem for Digraphs}

Now we discuss a combinatorial reciprocity theorem for directed graphs. Given a digraph $G$ on $N$, an acyclic coloring is a pair $( O, f)$ satisfying:
\begin{enumerate}
\item $O$ is an acyclic orientation of $G$
\item For every edge $(u,v) \in E(O)$, $f(u) \leq f(v)$.
\end{enumerate}
We modify the definition of descent. An $O$-descent is an edge $(u,v) \in E(G)$ where $(v,u) \in E(O)$.
Let $Y_G$ be the acyclic colorings. Then $\mathfrak{G}$ acts on $Y_G$. 
For $\mathfrak{g} \in \mathfrak{G}$, define:
\[\overline{\chi}(G, \mathfrak{G}, t, \mathbf{x}) = \sum\limits_{(O,f) \in \Fix_{\mathfrak{g}}(Y_G)} t^{\des(f)} \mathbf{x}^f. \]
Then $\overline{\chi}(G, \mathfrak{G}, t, \mathbf{x})$ is a class function, and hence is a quasisymmetric class function.

\begin{theorem}
Let $G$ be a digraph and let $\mathfrak{G} \subseteq \Aut(G)$. Then \[(-1)^{|N|} S \sgn \chi(G, \mathfrak{G}, t, \mathbf{x}) = \overline{\chi}(G, \mathfrak{G}, t, \mathbf{x}) \] and 
\[(-1)^{|N|} \sgn \chi(G, \mathfrak{G}, t, -x) = \overline{\chi}(G, \mathfrak{G}, t, x). \]
\label{thm:combrecdig}
\end{theorem}

\begin{proof} 
Let $\mathfrak{g} \in \mathfrak{G}$. Then 
\begin{align*}
    (-1)^{|N|} S \sgn \chi(G, \mathfrak{G}, t, \mathbf{x}; \mathfrak{g}) & = \sum\limits_{O \in \Fix_{\mathfrak{g}}(\mathcal{A}(G))} t^{\asc(O)} (-1)^{|N|} S \sgn \Omega(P_O, \mathfrak{G}_O, \mathbf{x}; \mathfrak{g}) \\ 
    & = \sum\limits_{O \in \Fix_{\mathfrak{g}}(\mathcal{A}(G))} t^{\asc(O)} \Omega(P_O^{\ast}, \mathfrak{G}_O, \mathbf{x}; \mathfrak{g}) \\ 
    & = \sum\limits_{O \in \Fix_{\mathfrak{g}}(\mathcal{A}(G))} t^{\asc(O)} \sum\limits_{f: (\overleftarrow{O}, f) \in \Fix_{\mathfrak{g}}(Y_G)} \mathbf{x}^f  \\
    & = \sum\limits_{(O, f) \in \Fix_{\mathfrak{g}}(Y_G)} t^{\des(O)} \mathbf{x}^f = \overline{\chi}(G, \mathfrak{G}, t, \mathbf{x}; \mathfrak{g}).
\end{align*}
where the first equality follows from Lemma \ref{lem:fundamental2}, and the second equality is due to Theorem \ref{thm:combrecdbl}. The third equality comes from observing that $f \in P_{P_O^{\ast}}$ if and only if $(\overleftarrow{O}, f)$ is an acyclic coloring. The fourth equality comes from observing that $\asc(O) = \des(\overleftarrow{O})$, and reindexing the summation by substituting $O$ with $\overleftarrow{O}$.
\end{proof}

\section{Flawlessness}
\label{sec:increasing}
In this section, we study the property of being $M$-increasing or effectively flawless.
Let $D$ be a double poset on a finite set $N$, and let $\mathfrak{G} \subseteq \Aut(D)$. For $\alpha \models |N|$, let $V_{\alpha, D}$ be the vector space with basis $X_{\alpha, D}$. Then $V_{\alpha, D}$ is a $\mathfrak{G}$-module.

Given $\alpha \leq \beta \models |N|$, we define $\mathfrak{G}$-invariant injective maps $\theta_{\alpha, \beta}: V_{\alpha, D} \to V_{\beta, D}$.
Given a $D$-set composition $C$ of type $\alpha$, we let \begin{equation} \theta_{\alpha, \beta}(C) = \sum\limits_{C' \in X_{\beta, D}: C' \geq C} C'. \label{eq:theta} \end{equation}
\begin{proposition}
Let $D$ be a double poset on a finite set $N$. Let $\alpha \leq \beta \leq \gamma $ be integer compositions of $|N|$. Then $\theta_{\alpha, \beta}$ is an injective $\mathfrak{G}$-invariant map. Moreover, we have $\theta_{\gamma, \beta} \circ \theta_{\alpha, \beta} = \theta_{\gamma, \alpha}$.
\label{prop:thetaprops}
\end{proposition}
\begin{proof}
We see that $\theta_{\alpha, \beta}$ is $\mathfrak{G}$-invariant. To see that the map is injective, it is enough to note that the map $f: V_{\beta, D} \to V_{\alpha, D}$ given by $f(C) = C_{\alpha}(C)$ is a section for $\theta_{\alpha, \beta}$. Hence $\theta_{\alpha, \beta}$ is injective and $f$ is surjective.

Let $\alpha \leq \beta \leq \gamma$, and let $C \in X_{\alpha, D}$. Then \begin{align*} \theta_{\beta, \gamma} \circ \theta_{\alpha, \beta}(C) & = \sum\limits_{C' \in X_{\beta, D}: C' \geq C} \theta_{\beta, \gamma}(C') \\
& = \sum\limits_{C' \in X_{\beta, D}: C' \geq C} \sum\limits_{C'' \in X_{\gamma, D}: C'' \geq C'} C'' \\
& = \sum\limits_{C'' \in X_{\gamma, D}: C'' \geq C} \sum\limits_{C' \in X_{\beta, D}: C'' \geq C' \geq C} C'' \\ 
& = \sum\limits_{C'' \in X_{\gamma, D}: C'' \geq C} C'' = \theta_{\alpha, \gamma}(C). \end{align*}
The penultimate equality follows from the fact that there is exactly one set composition $C'$ of type $\beta$ with the property that $C \leq C' \leq C''$. Moreover, if $C$ and $C''$ are $D$-set compositions, then $C'$ is as well. Hence the inner summation only involves one term.
\end{proof}

\begin{theorem}
Let $D$ be a double poset on a finite set $N$, and let $\mathfrak{G} \subseteq \Aut(D)$. Then $\Omega(D, \mathfrak{G}, \mathbf{x})$ is $M$-increasing and $\Omega(D, \mathfrak{G}, x)$ is effectively flawless.

Let $G$ be a digraph, and let $\mathfrak{H} \subseteq \Aut(G)$. For $k \in \mathbb{N}$, we have $[t^k]\chi(G, \mathfrak{H}, t, \mathbf{x})$ is $M$-increasing and $[t^k]\chi(G, \mathfrak{H}, t, x)$ is effectively flawless.
\label{thm:flawless}
\end{theorem}
\begin{proof}
Let $D$ be a double poset on a finite set $N$, and let $\mathfrak{G} \subseteq \Aut(D)$. For $\alpha \models |N|$, we see that $[M_{\alpha}] \Omega(D, \mathfrak{G}, \mathbf{x})$ is the character of the representation of $\mathfrak{G}$ on $V_{\alpha, D}$. Let $\alpha \leq \beta \models |N|$. Since $\theta_{\alpha, \beta}$ is injective and $\mathfrak{G}$-invariant, $V_{\alpha, D}$ is isomorphic to a submodule of $V_{\beta, D}$. Thus $[M_{\beta}] \Omega(D, \mathfrak{G}, \mathbf{x}) - [M_{\alpha}]\Omega(D, \mathfrak{G}, \mathbf{x})$ is the character corresponding to the complement of $\theta_{\alpha, \beta}(V_{\alpha, D})$ in $V_{\beta, D}$. Hence $\Omega(D, \mathfrak{G}, \mathbf{x})$ is $M$-increasing. 

Since $\Omega(D, \mathfrak{G}, \mathbf{x})$ is $M$-increasing and $\Omega(D, \mathfrak{G}, x)$ is the principal specialization of $\Omega(D, \mathfrak{G}, \mathbf{x})$, it follows from Proposition \ref{prop:global2} \ref{prop:ftohorbit} that $\Omega(D, \mathfrak{G}, x)$ is effectively flawless.

Let $G$ be a directed graph, and let $\mathfrak{H} \subseteq \Aut(G)$. By Theorem \ref{thm:feffectivedigraph}, we see that $[t^k]\chi(G, \mathfrak{H}, t, \mathbf{x}) $ is $F$-effective. Thus $[t^k]\chi(G, \mathfrak{H}, t, \mathbf{x}) $ is $M$-increasing. Since $\ps(\chi(G, \mathfrak{H}, t, \mathbf{x})) = \chi(G, \mathfrak{H}, t, x)$, it follows that $[t^k]\chi(G, \mathfrak{G}, t, x)$ is effectively flawless.
\end{proof}

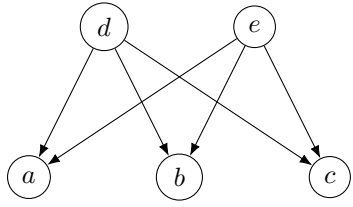
\begin{figure}
\begin{center}
\begin{tikzpicture}
  \node[circle, draw=black, fill=white] (b) at (1,0) {$b$};
  \node[circle, draw=black, fill=white] (d) at (0,2) {$d$};
  \node[circle, draw=black, fill=white] (a) at (-1,0) {$a$};
  \node[circle, draw=black, fill=white] (c) at (3,0) {$c$};
  \node[circle, draw=black, fill=white] (e) at (2,2) {$e$};
 \draw[-Latex] (d) -- (a);
  \draw[-Latex] (d) -- (b);
  \draw[-Latex] (e) edge (c);
  \draw[-Latex] (d) -- (c);
    \draw[-Latex] (e) -- (a);
    \draw[-Latex] (e) -- (b);

\end{tikzpicture}

\end{center}
\label{fig:notFincreasing}
\caption{A Hasse diagram.}
\end{figure}

Now we discuss counterexamples to extending Theorem \ref{thm:flawless} to studying coefficients in the $F$ basis and to studying the $h$-vector.
Let $D$ be the poset in Figure \ref{fig:notFincreasing}. We define $\leq_2$ to be the opposite partial order of  $\leq_1$. Hence $\Omega(D, x)$ counts strict $D$-partitions. We let $\mathfrak{G} = \mathbb{Z}_2 \times \mathbb{Z}_3$, viewed as the group $\langle (abc), (de) \rangle$. Let $\omega$ be a cubed root of unity. If we let $\sigma = (abc)$ and $\tau = (de)$, then Table \ref{table:character} is the character table for $\mathfrak{G}$:

\begin{table}
\begin{center}
$\begin{array}{c|cccccc}
& \{e\} & \{\sigma \} & \{\sigma^2 \} & \{\tau \} & \{\tau\sigma \} & \{\tau \sigma^2 \} \\
\hline
\chi_1 & 1 & 1 & 1 & 1 & 1 & 1\\
\chi_2 & 1 & \omega & \bar{\omega} & 1 & \omega & \bar{\omega} \\
\chi_3 & 1 & \bar{\omega} & \omega & 1 & \bar{\omega} & \omega \\
\chi_4 & 1 & 1 & 1 & -1 & -1 & -1\\
\chi_5 & 1 & \omega & \bar{\omega} & -1 & -\omega & -\bar{\omega} \\
\chi_6 & 1 & \bar{\omega} & \omega & -1 & -\bar{\omega} & -\omega
\end{array}$
\end{center}

\label{table:character}
\caption{The character table for $\mathbb{Z}_3 \times \mathbb{Z}_2$}
\end{table}

Then 
\begin{align*} \Omega(D, \mathfrak{G}, \mathbf{x}) & = \chi_1F_{3,2} + (\chi_2+\chi_3)(F_{1,2,2}+F_{2,1,2})+\chi_4F_{3,1,1} \\ & +(\chi_5+\chi_6)(F_{1,2,1,1}+F_{2,1,1,1}) + \chi_1 F_{1,1,1,2}+\chi_4F_{1,1,1,1,1}.\end{align*} 
Thus $\Omega(D, \mathfrak{G}, \mathbf{x})$ is \emph{not} $F$-increasing. We see that \[\sum_{m \geq 0} \Omega(D, \mathfrak{G}, m)t^m = \frac{\chi_1t^2+(2\chi_2+2\chi_3+\chi_4)t^3+(2\chi_5+2\chi_6+\chi_1)t^4+\chi_4t^5}{(1-t)^6}. \]
Then the $h$-vector is $(0, 0, \chi_1, 2\chi_2+2\chi_3+\chi_4, 2\chi_5+2\chi_6+\chi_1, \chi_4)$. We see that $h_2$ and $h_3$ are not comparable, so the $h$-vector is not strongly flawless.

Similarly, we can view the Hasse diagram in Figure \ref{fig:notFincreasing} as a digraph $G$ with $\mathfrak{G} = \mathbb{Z}_3 \times \mathbb{Z}_2$. Then $[t^0]\chi(G, \mathfrak{G}, \mathbf{x}) = \Omega(D, \mathfrak{G}, \mathbf{x})$. Thus we have an example of a directed graph where the chromatic quasisymmetric class function is not $F$-increasing, and where the $h$-vector of the chromatic polynomial class function is not effectively flawless.

\section{F-effectiveness and $h$-effectiveness}
\label{sec:feffect}
In this section, we state and prove several effectiveness theorems.
\begin{theorem}
Let $D$ be a locally special double poset on a finite set $N$, and let $\mathfrak{G} \subseteq \Aut(D)$. Then $\chiex{D}$ is $F$-effective.
\label{thm:dblposfeffect}
\end{theorem}
We prove Theorem \ref{thm:dblposfeffect} in the next subsection.
First, we focus on corollaries and other related theorems. First, we obtain the follow result from Proposition \ref{prop:global} \ref{prop:charpositive}.
\begin{corollary}
Let $D$ be a locally special double poset on a finite set $N$, and let $\mathfrak{G} \subseteq \Aut(D)$. Given an irreducible character $\psi$, we have $\langle \psi, \chih \rangle$ is $F$-positive.
\label{cor:characterfpos}
\end{corollary}

We also obtain the following theorem for the chromatic quasisymmetric class function of a digraph.
\begin{theorem}
Let $G$ be a digraph, and let $\mathfrak{G} \subseteq \Aut(G)$. For any $k \in \mathbb{N}$, we have $[t^k]\chi(G, \mathfrak{G}, \mathbf{x})$ is $F$-effective.

For any irreducible character $\psi$ of $\mathfrak{G}$, we have $\langle \psi, [t^k] \chi(G, \mathfrak{G}, \mathbf{x}) \rangle$ is $F$-positive.
\label{thm:feffectivedigraph}
\end{theorem}
The second statement follows from applying Proposition \ref{prop:global} \ref{prop:charpositive} to the first statement, so we focus on proving the first statement.
\begin{proof}
Let $\mathcal{T}$ be a transversal for the action of $\mathfrak{G}$ on $\mathcal{A}(G)$.
For each acyclic orientation $O \in \mathcal{T}$, we write $\Omega(P_O, \mathfrak{G}_O, \mathbf{x}) = \sum\limits_{\alpha \models |N|} \psi_{\alpha, P_O} F_{\alpha}$ where the characters $\psi_{\alpha, P_O}$ are effective.
Using Lemma \ref{lem:fundamental2}, we have \begin{align*} \chi(G, \mathfrak{G}, t, \mathbf{x}) & = \sum\limits_{O \in \mathcal{T}} t^{\asc(O)}  \Omega(P_O, \mathfrak{G}_{O}, \mathbf{x})\uparrow_{\mathfrak{G}_{O}}^{\mathfrak{G}} \\ 
 & \sum\limits_{O \in \mathcal{T}} t^{\asc(O)}\left(\sum\limits_{\alpha \models |N|} \psi_{\alpha, P_O} F_{\alpha}\right)\uparrow_{\mathfrak{G}_{O}}^{\mathfrak{G}} \\
 & =
 \sum\limits_{\alpha \models |N|} \left(\sum\limits_{O \in \mathcal{T}} t^{\asc(O)}\psi_{\alpha,P_O} \uparrow_{\mathfrak{G}_{O}}^{\mathfrak{G}}\right) F_{\alpha}. \end{align*}
\end{proof}

We also deduce results about $h$-effectiveness by applying Proposition \ref{prop:global2} \ref{prop:ftoheffect}.
\begin{corollary}
Let $D$ be a locally special double poset on a finite set $N$, and let $\mathfrak{G} \subseteq \Aut(D).$ Then $\Omega(D, \mathfrak{G}, x)$ is $h$-effective.

Let $G$ be a directed graph. Then $\chi(G, \mathfrak{G}, t, x)$ is $h$-effective.
\end{corollary}

\subsection{Proof of Theorem \ref{thm:dblposfeffect}}

For $\alpha \leq \beta \models |N|$, define $\theta_{\alpha, \beta}$ by Equation \eqref{eq:theta}.
Our first step is to prove the following proposition.
\begin{proposition}
Let $D$ be a double poset on a finite set $N$. Let $\alpha \leq \beta \models |N|$ and $\gamma \leq \beta$. Then we have $\theta_{\alpha, \beta}(V_{\alpha, D}) \cap \theta_{\gamma, \beta}(V_{\gamma, D}) = \theta_{\alpha \wedge \gamma, \beta}(V_{\alpha \wedge \gamma, D})$. 
\label{prop:dblposinter}
\end{proposition}

Fix a $D$-compatible linear order $\ell$, which exists by Lemma \ref{lem:compatible}.
Let $\alpha \leq \beta$ and $\gamma \leq \beta$. Fix $\vec{v} \in \theta_{\alpha, \beta}(V_{\alpha, D}) \cap \theta_{\gamma, \beta}(V_{\gamma, D}).$ Consider $C \in X_{\beta, D} \cap \theta_{\alpha, \beta}(V_{\alpha, D})$. Then $C_{\alpha}(C) \in X_{\alpha, D}$. Let $C^1 = C_{\beta}(\ell(C_{\alpha}(C))).$ We see that $C_{\alpha}(C) \leq C^1$, so if $C^1$ contained any inversions, then so does $C_{\alpha}(C)$. Since $C^1 \leq \ell(C_{\alpha}(C))$, we know that the union of the first several blocks of $C^1$ is the union of the first several blocks of $\ell(C_{\alpha}(C)).$ Thus $C^1$ is a $D$-set composition.

Since $C_{\alpha}(C) \leq C$ and $C^1 \leq \ell(C_{\alpha}(C))$, we have $\ell(C_1) \leq_{\ell} \ell(C_{\alpha}(C)) <_{\ell} \ell(C)$. Thus $C^1 \leq_{\ell} C$. Moreover, since $\vec{v} \in \theta_{\alpha, \beta}(V_{\alpha, D})$, we know $\vec{v} = \theta_{\alpha, \beta}(\vec{w})$ for some $\vec{w} \in V_{\alpha, D}$. Since $C_{\alpha}(C) \leq C^1$, it follows from the definition of $\theta_{\alpha, \beta}$  that $[C^1]\vec{v} = [C_{\alpha}(C)]\vec{w}$. Similarly, since $C_{\alpha}(C) \leq C$, we have $[C]\vec{v} = [C_{\alpha}(C)]\vec{w}$. Thus $[C^1]\vec{v} = [C]\vec{v}$. 

Let $C^2 = C_{\beta}(\ell(C_{\gamma}(C^1))).$ By similar reasoning, $C^2 \leq_{\ell} C^1$ and $[C^2]\vec{v} = [C^1]\vec{v} = [C]\vec{v}$.
For every integer $k$, we define 
\[ C^{k+1} = \begin{cases}
 C_{\beta}(\ell(C_{\alpha}(C^k))) & k \mbox{ is even} \\
 C_{\beta}(\ell(C_{\gamma}(C^k))) & k \mbox{ is odd}
\end{cases} 
\]
Thus we obtain a sequence of compositions $C^k$ such that $C^{k+1} \leq_{\ell} C^k$ and $[C^k]\vec{v} = [C]\vec{v}$ for all $k$. 

Let $m$ be the first integer where $C^{m+1} = C^m$. Define $f(C) = C_{\alpha \wedge \gamma}(\ell(C^m))$. 
\begin{lemma}
Let $\vec{v} \in \theta_{\alpha, \beta}(V_{\alpha, D}) \cap \theta_{\gamma, \beta}(V_{\gamma, D}).$ Consider $C \in X_{\beta, D}$. Fix $f(C)$ as defined above. Then $f(C) \in X_{\alpha \wedge \gamma, D}$.

Moreover, given $D$ such that $D \geq f(C)$, we have $f(D) = f(C).$
\label{lem:needed}
\end{lemma}
Before we give the proof, we define another operation on set compositions. Given $C \models N$, and $S \subset N$, the set composition $C \cap S$ is defined by taking $(C_1 \cap S, C_2 \cap S , \cdots ,C_{\ell(C)} \cap S)$, and then deleting any coordinates corresponding to empty intersections. For instance, $123|456|789 \cap \{1, 7,8 \} = 1|78$.
\begin{proof}
 We show that $f(C)$ is a $D$-set composition. Write $f(C) = C'_1 | \cdots | C'_r$. Since $f(C) \leq C^m$, the set $C'_1 \cup \cdots \cup C'_i$ is a union of blocks of $C^m$, and hence is a $\leq_1$-order ideal. Thus, it suffices to show that that $C'_i$ does not contain any inversions. Define $D' = C_{\alpha}(C^m) \cap C'_i$, which is a set composition of $C'_i$. Let $\ell'$ be the restriction of $\ell(C^m)$ to $C'_i$. Then $\ell'$ is lexicographically least in $C'_i$; if there exists $\tau <_{\ell} \ell'$, then we could modify $\ell(C^m)$ by replacing $\ell(C^m)|_{C'_i}$ with $\tau$ and obtain a new refinement of $C^m$ that is lexicographically smaller than $\ell(C^m)$. Since each block $D'_j$ of $D'$ contains no inversions, by Proposition \ref{prop:increasing}, $\ell'|_{D'_j}$ is strictly increasing. If we let $E' = C_{\gamma}(C^m) \cap C'_i$, then by similar reasoning $\ell'$ is strictly increasing when restricted to each block of $E'$. We see that the finest common coarsening of $D'$ and $E'$ is the set composition $(C_i)$.
 Thus $\ell(C^m)$ is strictly increasing on all of $C'_i$. By Proposition \ref{prop:increasing}, it follows that $D$ has no inversions in $C'_i$. Therefore $C'$ is a $D$-set composition. We also see that $C^m = C_{\gamma}(\ell(f(C))).$

Now let $D \geq f(C)$. For every integer $k$, we define 
\[ D^{k+1} = \begin{cases}
 C_{\beta}(\ell(C_{\alpha}(D^k))) & k \mbox{ is even} \\
 C_{\beta}(\ell(C_{\gamma}(D^k))) & k \mbox{ is odd}
\end{cases} 
\]
 Thus we obtain a sequence of compositions $D^k$ such that $D^{k+1} \leq_{\ell} D^k$ and $[D^k]\vec{v} = [D]\vec{v}$ for all $k$. We observe that $D^{k+1}$ involves permuting the elements from blocks of $D^k$ that belong to the same block of $C_{\alpha}(D^k)$ (or $C_{\gamma}(C^k)$). Hence $C_{\alpha \wedge \gamma}(D^{k+1}) = C_{\alpha \wedge \gamma}(D^k) = f(C)$ for all $k$. Thus $D^{k+1} \geq f(C)$ for all $k$.

If we let $m$ be the first integer for which $D^{m+1} = D^m$, then we have $f(D) = C_{\alpha \wedge \gamma}(D^m) = f(C)$.
\end{proof}
\begin{proof}[Proof of Proposition \ref{prop:dblposinter}]
Let $\vec{v} \in \theta_{\alpha, \beta}(V_{\alpha, D}) \cap \theta_{\gamma, \beta}(V_{\gamma, D}).$
For any $C \in X_{\beta, D}$, let $m(C) = C_{\beta}(\ell(f(C)))$. Then we have $\vec{v}_C = \vec{v}_{m(C)}$. Moreover, for any $C' \geq f(C)$, by Lemma \ref{lem:needed}, we have $f(C) = f(C')$. Hence $m(C') = m(C)$ and  $[C]\vec{v} = [C']\vec{v}$. Given $D \in X_{\alpha \wedge \gamma, D}$, we define $[D]\vec{w} = [C]\vec{v}$ for any $C \in X_{\beta, D}$ such that $f(C) = D$. By what we have just observed, $[D]\vec{w}$ does not depend upon the choice of $C$. Also, for any $C \in X_{\beta, D}$, we know that $f(C) \in X_{\alpha \wedge \gamma, D}$ by Lemma \ref{lem:needed}.
Thus $\theta_{\alpha \wedge \gamma, \beta}(\vec{w}) = \vec{v}$. Hence $\vec{v} \in \theta_{\alpha \wedge \gamma, \beta}(V_{\alpha \wedge \gamma, D})$.
\end{proof}

For $\alpha \models |N|$, we define $W_{\alpha, D}$ to be the complement in $V_{\alpha, D}$ of the module \[\vspan \{ \theta_{\beta, \alpha}(X_{\beta, D}): \beta < \alpha \}\]
\begin{proposition}
Let $D$ be a locally special double poset on a finite set $N$. Let $\alpha \leq \beta \models |N|$ and $\gamma \leq \beta$. Then we have $\theta_{\alpha, \beta}(W_{\alpha, D}) \cap \theta_{\gamma, \beta}(W_{\gamma, D}) = \{\vec{0} \}$.
\end{proposition}
\begin{proof} 
Let $\vec{v} \in \theta_{\alpha, \beta}(W_{\alpha, D}) \cap \theta_{\gamma, \beta}(W_{\gamma, D}).$ Then $\vec{v} = \theta_{\alpha, \beta}(\vec{y})$ for some $\vec{y} \in W_{\alpha, D}$.
By Proposition \ref{prop:dblposinter} and Proposition \ref{prop:thetaprops}, we have $\vec{v} \in \theta_{\alpha \wedge \gamma, \beta}(V_{\alpha \wedge \gamma, D}) = \theta_{\alpha, \beta} \circ \theta_{\alpha \wedge \gamma, \alpha}(V_{\alpha \wedge \gamma, D}).$ Thus $\vec{v} = \theta_{\alpha, \beta}(\vec{z})$ for some $\vec{z} \in \theta_{\alpha \wedge \gamma, \alpha}(V_{\alpha \wedge \gamma, D}).$ Since $\theta_{\alpha, \beta}$ is injective, we have $\vec{y} = \vec{z}$. However, by the definition of $W_{\alpha, D}$, it follows that $\vec{y} = \vec{z} = \vec{0}$, and thus $\vec{v} = \vec{0}$. Hence $\theta_{\alpha, \beta}(W_{\alpha, D}) \cap \theta_{\gamma, \beta}(W_{\gamma, D}) = \{\vec{0}\}$. \end{proof}

\begin{proof}[Proof of Theorem \ref{thm:dblposfeffect}]
Let $D$ be a locally special poset on $N$ with $\mathfrak{G} \subseteq \Aut(D)$.
Let us write $\chiex{D} = \sum\limits_{\alpha \models |N|} \psi_{\alpha, D} F_{\alpha}$ where the $\psi_{\alpha, D}$ are virtual characters. Then we have $\chialphaex{D} = \sum\limits_{\beta \leq \alpha} \psi_{\beta, D}$. We claim that $\psi_{\beta, D}$ is the character of $W_{\beta, D}$. It suffices to show that \begin{equation} V_{\alpha, D} = \bigoplus_{\beta \leq \alpha} \theta_{\beta, \alpha}(W_{\beta, D}) \simeq \bigoplus_{\beta \leq \alpha} W_{\beta, D} \label{eq:iso} \end{equation} as $\mathfrak{G}$-modules for all $\alpha \models |N|$. 

We prove Equation \eqref{eq:iso} by induction on $\ell(\alpha)$. We see that \[\vspan\{\theta_{\beta, \alpha}(W_{\beta, D}): \beta < \alpha \} \subseteq \vspan \{\theta_{\beta, \alpha}(X_{\beta, D}): \beta < \alpha \}. \] By induction, we have \[\theta_{\beta, \alpha}(V_{\beta, D}) = \theta_{\beta, \alpha}(\bigoplus_{\gamma \leq \beta} \theta_{\gamma, \beta}(W_{\gamma, D })) = \bigoplus_{\gamma \leq \beta} \theta_{\gamma, \alpha}(W_{\gamma, D}).
\]
The last equality uses the fact that $\theta_{\beta, \alpha}$ is $\mathfrak{G}$-invariant, and the property that $\theta_{\beta, \alpha} \circ \theta_{\gamma, \beta} = \theta_{\gamma, \alpha}$.
Thus we have \[\vspan \{\theta_{\beta, \alpha}(X_{\beta, D}): \beta < \alpha \} = \vspan \{\theta_{\beta, \alpha}(W_{\beta, D}): \beta < \alpha \} = \bigoplus_{\beta < \alpha} \theta_{\beta, \alpha}(W_{\beta, D}).\] 
Combined with the definition of $W_{\alpha, D}$, we get \[V_{\alpha, D} = \bigoplus_{\beta \leq \alpha} \theta_{\beta, \alpha}(W_{\beta, D}) \simeq \bigoplus_{\beta \leq \alpha} W_{\beta, D} \]
where the last isomorphism comes from the fact that $\theta_{\beta, \alpha}$ is $\mathfrak{G}$-invariant and injective.
\end{proof}

\section{Orbital Invariants}
\label{sec:orbinv}
In this section we define orbital quasisymmetric function invariants. In the case of double posets, the resulting invariant was already studied by Grinberg \cite{grinberg}. These are quasisymmetric functions whose coefficients count the number of orbits of a group action. Due to Burnside's Lemma, these invariants can be computed as $\langle 1, \chi \rangle$, where $1$ is the trivial character, and $\chi$ is the quasisymmetric class function. As a result, we see that we derive many results for our orbital invariants from the class functions.

We define the \emph{orbital quasisymmetric $D$-partition enumerator} by 
\[\Omega^O(D, \mathfrak{G}, \mathbf{x}) = \frac{1}{|\mathfrak{G}|} \sum\limits_{\mathfrak{g} \in \mathfrak{G}} \Omega(D, \mathfrak{G}, \mathbf{x}; \mathfrak{g}) \] 
and the \emph{orbital order polynomial} by 
\[\Omega^O(D, \mathfrak{G}, x) = \frac{1}{|\mathfrak{G}|} \sum\limits_{\mathfrak{g} \in \mathfrak{G}} \Omega(D, \mathfrak{G}, x; \mathfrak{g}). \]

We also define the \emph{orbital chromatic quasisymmetric function} of a digraph $G$ by 
\[\chi^O(G, \mathfrak{G}, t, \mathbf{x}) = \frac{1}{|\mathfrak{G}|} \sum\limits_{\mathfrak{g} \in \mathfrak{G}} \chi(G, \mathfrak{G}, t, \mathbf{x}; \mathfrak{g}) \] and the \emph{orbital chromatic polynomial} by 
\[\chi^O(G, \mathfrak{G}, t, x) = \frac{1}{|\mathfrak{G}|} \sum\limits_{\mathfrak{g} \in \mathfrak{G}} \chi(G, \mathfrak{G}, t, x; \mathfrak{g}) \]

The following results follow from Burnside's Lemma. Property 2 is Grinberg's definition of the orbital $D$-partition enumerator.
\begin{proposition}
Let $D$ be a double poset on a finite set $N$ and let $\mathfrak{G} \subseteq \Aut(D)$. Let $G$ be a digraph on $N$, and let $\mathfrak{H} \subseteq \Aut(G)$.
\begin{enumerate}
\item We have $[M_{\alpha}]\Omega^O(D, \mathfrak{G}, \spe{x}) = |X_{\alpha, D} / \mathfrak{G}|$.

\item Let $\mathcal{T}$ be a transversal for $\mathfrak{G}$ acting on $P_D$. Then $\Omega^O(D, \mathfrak{G}, \mathbf{x}) = \sum\limits_{f \in \mathcal{T}} \mathbf{x}^f$.

\item For $n \in \mathbb{N}$, we have $\Omega^O(D, \mathfrak{G}, n) = |X_{n, D} / \mathfrak{G}|$.

\item Let $\mathcal{T}$ be a transversal for $\mathfrak{H}$ acting on $C_G$. Then $\Omega^O(G, \mathfrak{H}, \mathbf{x}) = \sum\limits_{f \in \mathcal{T}} \mathbf{x}^f$.

\item For $n \in \mathbb{N}$, and $k > 0$, we have $[t^k]\chi^O(G, \mathfrak{H}, t, n) = |C_{k, n, G} / \mathfrak{H}|$.

\end{enumerate}
\end{proposition}
\begin{proof}
We will prove the first and second claim. The rest are similar. For the first, we observe that \[[M_{\alpha}] \Omega^O(D, \mathfrak{G}, \mathbf{x}) = \frac{1}{|\mathfrak{G}|} \sum\limits_{\mathfrak{g} \in \mathfrak{G}} [M_{\alpha}]\Omega(D, \mathfrak{G}, \mathbf{x}; \mathfrak{g}) = \frac{1}{|\mathfrak{G}|} \sum\limits_{\mathfrak{g} \in \mathfrak{G}} |\Fix_{\mathfrak{g}}(X_{\alpha, D})|.\]

For the second identity, we note that for any $f: N \to \mathbb{N}$ and $\sigma \in \mathfrak{S}_N$, we have $\mathbf{x}^{\sigma f} = \mathbf{x}^f$. We see that
\begin{align*} \frac{1}{|\mathfrak{G}|} \sum\limits_{\mathfrak{g} \in \mathfrak{G}} \Omega(D, \mathfrak{G}, \mathbf{x}; \mathfrak{g}) & = \frac{1}{|\mathfrak{G}|} \sum\limits_{\mathfrak{g} \in \mathfrak{G}} \sum\limits_{f \in \Fix_{\mathfrak{g}}(X_D)} \mathbf{x}^f \\
&= \frac{1}{|\mathfrak{G}|} \sum\limits_{f \in X_D} \mathbf{x}^f \sum\limits_{\mathfrak{g} \in \mathfrak{G}_f} 1 \\
&= \sum\limits_{f \in X_D} \frac{1}{|\mathfrak{G}(f)|}\mathbf{x}^f \\
&= \sum\limits_{h \in \mathcal{T}} \mathbf{x}^h \sum\limits_{f \in \mathfrak{G}(h)} \frac{1}{|\mathfrak{G}_f|}.
\end{align*}
The second equality is just a rearrangment of terms. The third equality follows from the Orbit-Stabilizer Theorem. The last equality involves splitting the summation into a double sum, and recognize that $\mathbf{x}^f$ is constant on orbits. Finally, the inner summation simplifies to 1.
\end{proof}

We also obtain several facts about the coefficients of the orbital invariants with respect to various bases. For polynomials $f(t)$ and $g(t)$, we write $f(t) \leq_t g(t)$ if $g(t) - f(t) \in \mathbb{N}[t]$. 
\begin{corollary}
Let $D$ be a double poset on $N$. Given $\mathfrak{G} \subseteq \Aut(D),$ we have $\Omega^O(D, \mathfrak{G}, \mathbf{x})$ is $M$-increasing, and $\Omega^O(D, \mathfrak{G}, x)$ is strongly flawless.
\end{corollary}
\begin{proof}
Let $D$ be a double poset on $N$, and let $\mathfrak{G} \subseteq \Aut(D)$. 
Since $\Omega(D, \mathfrak{G}, \mathbf{x})$ is $M$-increasing, then if we take the inner product with the trivial character, Proposition \ref{prop:global} \ref{prop:charincreasing} implies that that $\Omega^O(D, \mathfrak{G}, \mathbf{x})$ is $M$-increasing. Also, Proposition \ref{prop:global2} \ref{prop:flawlessorb} implies that $\Omega^O(D, \mathfrak{G}, x)$ is strongly flawless.

Suppose that $D$ is locally special. Then by Corollary \ref{cor:characterfpos}, applied with $\psi$ being the trivial character, we see that $\Omega^O(D, \mathfrak{G}, \mathbf{x})$ is $F$-positive. Since $\Omega^O(D, \mathfrak{G}, x) = \ps(\Omega^O(D, \mathfrak{G}, \mathbf{x}))$, it follows from Proposition \ref{prop:global2} \ref{prop:flawlessorb} that $\Omega^O(D, \mathfrak{G}, x)$ is $h$-positive.

Let $G$ be a directed graph, and let $\mathfrak{H} \subseteq \Aut(G)$. Fix $k \in \mathbb{N}$. By Theorem \ref{thm:feffectivedigraph}, we see that $[t^k]\chi^O(G, \mathfrak{H}, \mathbf{x})$ is $F$-positive. Hence it is also $M$-increasing. By principal specialization and Proposition \ref{prop:global2} \ref{prop:flawlessorb}, we see that $[t^k]\chi^O(G, \mathfrak{H}, x)$ is $h$-positive, and $[t^k]\chi^O(G, \mathfrak{H}, x)$ is strongly flawless.

\end{proof}

\subsection{Orbital Combinatorial Reciprocity Results}

We can also obtain a combinatorial reciprocity for the orbital $D$-partition enumerator, although it involves the notion of coeven $D$-partition. Given a group $\mathfrak{G} \subseteq \mathfrak{S}_N$, and an action of $\mathfrak{G}$ on a set $X$, let we say that an element $x \in X$ is $\mathfrak{G}$-coeven if the stabilizer subgroup $\mathfrak{G}_x$ is a subgroup of the alternating group $\mathfrak{A}_N$. Let $X^+$ be the set of $\mathfrak{G}$-coeven elements. Then $\mathfrak{G}$ acts on $X^+$.

The following Lemma is essentially a restatement of Lemma 7.7 in Grinberg \cite{grinberg}, although our proof is slightly different.
\begin{lemma}
Let $\mathfrak{G} \subseteq \mathfrak{S}_N$ be a group acting on a finite set $X$.
\[ |X^+/\mathfrak{G}| = \frac{1}{|\mathfrak{G}|}\sum\limits_{\mathfrak{g} \in \mathfrak{G}} \sgn(\mathfrak{g}) |\Fix_{\mathfrak{g}}(X)|\]
\label{lem:coeven}
\end{lemma}
\begin{proof}
We have
\begin{align*} \frac{1}{|\mathfrak{G}|}\sum\limits_{\mathfrak{g} \in \mathfrak{G}} \sgn(\mathfrak{g}) |\Fix_{\mathfrak{g}}(X)| & = \frac{1}{|\mathfrak{G}|}\sum\limits_{\mathfrak{g} \in \mathfrak{G}} \sum\limits_{x \in X: \mathfrak{g}x = x} \sgn(\mathfrak{g}) \\
& = \sum\limits_{x \in X} \frac{1}{|\mathfrak{G}|}\sum\limits_{\mathfrak{g} \in \mathfrak{G}_x} \sgn(\mathfrak{g}).
\end{align*}
Suppose that $\mathfrak{G}_x \not\subseteq \mathfrak{A}_N$. Then $\mathfrak{H} = \mathfrak{A}_N \cap \mathfrak{G}$ is a normal subgroup of $\mathfrak{G}$ of index $2$. Thus half the elements $\mathfrak{g}$ of $\mathfrak{G}_x$ are even, and half are odd, and these elements have opposite signs under $\sgn$. Hence the inner sum is zero in that case.

Hence we are left with those $x$ for which $\mathfrak{G}_x \subseteq \mathfrak{A}_N$. Then we obtain:
\begin{align*}
    \frac{1}{|\mathfrak{G}|}\sum\limits_{\mathfrak{g} \in \mathfrak{G}} \sgn(\mathfrak{g}) |\Fix_{\mathfrak{g}}(X)| & = \sum\limits_{x \in X^+} \frac{1}{|\mathfrak{G}|}\sum\limits_{\mathfrak{g} \in \mathfrak{G}_x} \sgn(\mathfrak{g}) \\ 
    & = \frac{1}{|\mathfrak{G}|}\sum\limits_{\mathfrak{g} \in \mathfrak{G}} \sum\limits_{x \in \Fix_{\mathfrak{g}}(X^+)} 1 \\
    & = |X^+/\mathfrak{G}|.
\end{align*}
\end{proof}

We define the \emph{$\mathfrak{G}$-coeven quasisymmetric function} by 
\[\Omega^+(D, \mathfrak{G}, \mathbf{x}) = \frac{1}{|\mathfrak{G}|} \sum\limits_{\mathfrak{g} \in \mathfrak{G}} \sgn(\mathfrak{g}) \Omega(D, \mathfrak{G}, \mathbf{x}; \mathfrak{g}) \] 
and the \emph{orbital polynomial} by 
\[\Omega^+(D, \mathfrak{G}, x) = \frac{1}{|\mathfrak{G}|} \sum\limits_{\mathfrak{g} \in \mathfrak{G}} \sgn(\mathfrak{g}) \Omega(D, \mathfrak{G}, x; \mathfrak{g}). \]
By using Lemma \ref{lem:coeven}, we obtain the following results.
\begin{proposition}
Let $D$ be a double poset on a finite set $N$ and let $\mathfrak{G} \subseteq \Aut(D)$. Let $G$ be a digraph on $N$, and let $\mathfrak{H} \subseteq \Aut(G)$.
\begin{enumerate}
\item We have $[M_{\alpha}]\Omega^+(D, \mathfrak{G}, \spe{x}) = |X^+_{\alpha, D} / \mathfrak{G}|$.

\item Let $\mathcal{T}$ be a transversal for $\mathfrak{G}$ acting on $P_D^+$. Then $\Omega^+(D, \mathfrak{G}, \spe{x}) = \sum\limits_{f \in \mathcal{T}} \mathbf{x}^f$.

\item For $n \in \mathbb{N}$, we have $\Omega^+(D, \mathfrak{G}, n) = |X^+_{n, D} / \mathfrak{G}|$.

\item If $\mathcal{T}$ is a transversal for $\mathfrak{H}$ acting on $C_D^+$, then $\chi^+(G, \mathfrak{H}, t, n) = \sum\limits_{f \in \mathcal{T}} t^{\asc(f)} \mathbf{x}^f$.

\item For $n \in \mathbb{N}$, and $k > 0$, we have $[t^k]\chi^+(G, \mathfrak{H}, t, n) = |C^+_{k, n, G} / \mathfrak{H}|$.

\end{enumerate}
\end{proposition}

Now we discuss the Combinatorial Reciprocity Theorem for orbital quasisymmetric functions. The first result is Theorem 4.7 of Grinberg \cite{grinberg}.
\begin{theorem}
Let $D$ be a locally special double poset on a finite set $N$, and let $\mathfrak{G} \subseteq \Aut(D)$.
Then we have the following identities.
\begin{enumerate}
    \item $(-1)^{|N|} S \Omega^{O}(D, \mathfrak{G}, \mathbf{x}) =  \Omega^+(D^{\ast}, \mathfrak{G}, \mathbf{x}).$
    \item $(-1)^{|N|} S \Omega^{+}(D, \mathfrak{G}, \mathbf{x}) =  \Omega^O(D^{\ast}, \mathfrak{G}, \mathbf{x}).$
    \item $(-1)^{|N|}\Omega^O(D, \mathfrak{G}, -x) = \Omega^+(D^{\ast}, \mathfrak{G}, x)$.
    \item $(-1)^{|N|}\Omega^+(D, \mathfrak{G}, -x) = \Omega^O(D^{\ast}, \mathfrak{G}, x)$.
\end{enumerate}
\label{thm:orbcombrec2}
\end{theorem}
Now we discuss the Combinatorial Reciprocity Theorem for orbital chromatic quasisymmetric functions.
\begin{theorem}
Let $G$ be a digraph on a finite set $N$, and let $\mathfrak{G} \subseteq \Aut(G)$.

Then we have the following identities:
\begin{enumerate}
    \item $(-1)^{|N|} S \chi^{O}(G, \mathfrak{G}, t, \mathbf{x}) =  \overline{\chi}^+(G, \mathfrak{G}, t, \mathbf{x}).$
    \item $(-1)^{|N|} S \chi^{+}(G, \mathfrak{G}, t, \mathbf{x}) =  \overline{\chi}^O(G, \mathfrak{G}, t, \mathbf{x}).$
    \item $(-1)^{|N|}\chi^O(G, \mathfrak{G},t,  -x) = \overline{\chi}^+(G, \mathfrak{G},t, x)$.
    \item $(-1)^{|N|}\chi^+(G, \mathfrak{G},t, -x) = \overline{\chi}^O(G, \mathfrak{G},t, x)$.
\end{enumerate}
\label{thm:orbcombrecgraph}
\end{theorem}

\begin{proof}
All formulas are proven in a similar manner, so we only prove the first formula. Let $G$ be a digraph on a finite set $N$, and let $\mathfrak{G} \subseteq \Aut(G)$. By Theorem \ref{thm:combrecdig}, we have
\begin{align*}
    (-1)^{|N|} S \chi^O(G, \mathfrak{G}, t, \mathbf{x}) & = \frac{(-1)^{|N|}}{|\mathfrak{G}|} \sum\limits_{\mathfrak{g} \in \mathfrak{G}} S\chi(G, \mathfrak{G}, t, \mathbf{x}; \mathfrak{g}) \\
    & = \frac{1}{|\mathfrak{G}|} \sum\limits_{\mathfrak{g} \in \mathfrak{G}} \sgn(\mathfrak{g})\overline{\chi}(G, \mathfrak{G}, t, \mathbf{x}; \mathfrak{g}) \\
   & = \overline{\chi}^+(G, \mathfrak{G}, t, \mathbf{x}).
\end{align*}
\end{proof} 

\section{Future Directions}
\label{sec:future}

First, we note that Stapledon \cite{stapledon} defines a different generalization of the $h^{\ast}$-vector. His work involves a group $\mathfrak{G}$ acting on the lattice points of a polytope. Given a quasipolynomial $p(x)$, whose coefficients are characters, he defines 
\[ \sum\limits_{n \geq 0} p(n) t^n = \frac{h^{\ast}(t)}{(1-t)\det[I-t\rho]}\] 
The $h^{\ast}$-vector is defined for any class function that takes on values in the ring of quasipolynomials, while our $h$-vector is only defined for class functions in the ring of polynomials, so the $h^{\ast}$-vector is a more general invariant. When we restrict to polynomials, the $h^{\ast}$-vector is different than the $h$-vector.
Is the order polynomial class function of a double poset $h^{\ast}$-effective? Is the chromatic polynomial of a digraph $h^{\ast}$-effective?

There are other bases of quasisymmetric functions. A very recent basis is the basis of quasisymmetric power sums $\Psi_{\alpha}$, introduced in \cite{powersums}. It has been shown that the $\spe{P}$-partition enumerator, and the chromatic quasisymmetric function of a digraph are both $\Psi$-positive \cite{alexandersson}. Is the $D$-partition quasisymmetric class function $\Psi$-effective? Is the chromatic quasisymmetric class function $\Psi$-effective? 

We also would like to have a better description of $\Omega^O(D, \mathfrak{G}, \mathbf{x})$ in the $F$ basis. We know the coefficients are positive. What do they count? 

On a similar note, is the $f$-vector $\Omega^O(D, \mathfrak{G}, x)$ effectively unimodal? This would mean that there exists an $i$ such that $f_j \leq_{\mathfrak{G}} f_k$ for $k \leq i$, and $f_j \geq_{\mathfrak{G}} f_k$ for $i \leq j$. This is still an open question even for trivial group actions. Our example in Figure \ref{fig:notFincreasing} shows that the $h$-vector for a locally special double poset can fail to be effectively unimodal. We can obtain the $h$-vector of $\Omega^O(D, \mathfrak{G}, x)$ by taking the coefficients of $\chi_1$ in the $h$-vector of $\Omega(D, \mathfrak{G}, x)$. Doing so results in the sequence $(0,0,1,0,1,0)$ which fails to be unimodal.

Finally, we could consider proving our results using the theory of combinatorial Hopf monoids, as studied by Aguiar and Mahajan \cite{aguiar-mahajan-1}. We will pursue this idea in a subsequent paper, using the theory of Hopf monoids to prove $F$-effectiveness and combinatorial reciprocity results for quasisymmetric class function invariants associated other combinatorial Hopf monoids.

\bibliographystyle{amsplain}  
\bibliography{orbital}

\end{document}